\documentclass[]{article}
\usepackage{graphicx}
\usepackage[utf8]{inputenc}
\usepackage[british]{babel}
\usepackage{setspace}
\usepackage{amsmath}
\usepackage{mathrsfs}
\usepackage{amssymb}
\usepackage{booktabs}
\usepackage{tikz-cd}
\usepackage{verbatim}
\usepackage{epigraph}
\usepackage{geometry}
\usepackage{bookmark}
\usepackage{enumitem}
\usepackage{mathtools}
\usepackage{amsthm}

\usepackage{cite}

\theoremstyle{plain} \newtheorem{theorem}{Theorem}[section]
\theoremstyle{definition} \newtheorem{definition}[theorem]{Definition}
\theoremstyle{plain} \newtheorem{proposition}[theorem]{Proposition}
\theoremstyle{plain} \newtheorem{lemma}[theorem]{Lemma}
\theoremstyle{plain} \newtheorem{corollary}[theorem]{Corollary}
\theoremstyle{remark} \newtheorem{remark}[theorem]{Remark}
\theoremstyle{remark} 
\theoremstyle{definition} \newtheorem{example}[theorem]{Example}

\newcommand{\numberset}{\mathbb}
\newcommand{\N}{\numberset{N}}

\newcommand{\C}{\numberset{C}}

\newcommand{\Z}{\numberset{Z}}

\newcommand{\lieg}{\mathfrak{g}}

\newcommand{\liel}{\mathfrak{l}}

\DeclareMathOperator{\orb}{\mathcal{O}}
\DeclareMathOperator{\ind}{Ind}

\DeclareMathOperator{\bir}{Bir}
\DeclareMathOperator{\wbir}{WBir}
\DeclareMathOperator{\codim}{codim}

\DeclareMathOperator{\Lie}{Lie}

\DeclareMathOperator{\Ad}{Ad}

\begin{document}

\title{Birational sheets in reductive groups}

\author{Filippo Ambrosio\\
Dipartimento di Matematica ``Tullio Levi-Civita''\\
Torre Archimede - via Trieste 63 - 35121 Padova - Italy\\
ambrosio@math.unipd.it
}
\date{}

\maketitle

\begin{abstract}
We define the group analogue of birational sheets, a construction performed by Losev for reductive Lie algebras. For $G$ semisimple simply connected, we describe birational sheets in terms of Lusztig-Spaltenstein induction and we prove that they form a partition of $G$, and that they are unibranch varieties with smooth normalization by means of a local study.
\end{abstract}

\section{Introduction}
The action of a connected algebraic group $G$ on an algebraic variety $X$ can be studied by gathering orbits in finitely many families to deduce properties shared by orbits in the same collection. One way of grouping orbits together is to form sheets, i.e. maximal irreducibile subsets of $X$ consisting of equidimensional orbits.
In \cite{BK79}, Borho and Kraft studied sheets for the adjoint action of a semisimple connected group $G$ on its Lie algebra $\mathfrak{g}$: the authors considered non-nilpotent orbits as deformations of nilpotent ones of the same dimension to compare the $G$-module structure of their ring of regular functions.
In the same paper, sheets and their closures were described set-theoretically as unions of decomposition classes.
The latter form a partition of $\mathfrak{g}$ into finitely-many, irreducible, smooth, $G$-stable, locally closed subsets, see \cite{Broer_Lectures, broer_1998}.
In \cite{Borho}, Borho described sheets of $\mathfrak{g}$ in terms of Lusztig-Spaltenstein induction.
If $\mathfrak{g} = \mathfrak{sl}_n(\C)$, sheets are parametrized by partitions of $n$ and any two distinct sheets have trivial intersection. This does not hold in general: for example, all simple non-simply laced Lie algebras present two sheets of subregular elements intersecting non-trivially. For $\mathfrak{g}$ simple and classical, all sheets are smooth (see \cite{imhof}), but this does not extend to exceptional Lie algebras (the list of smooth sheets is to appear in \cite{bulois}).
In \cite[§4]{Lo2016}, Losev applied the theory of universal Poisson deformations of conical symplectic singularities to define birational sheets of $\mathfrak{g}$. He proved that, unlike sheets, birational sheets form a partition of $\mathfrak{g}$, they are smooth up to a bijective normalization and all birational sheets of $\mathfrak{g}$ simple and classical are smooth.
Furthermore, the $G$-module structure of the ring of functions of adjoint orbits in the same birational sheet is preserved and Losev conjectures that birational sheets can be parametrized by this invariant.
The group analogue of decomposition classes, called Jordan classes, first appeared in Lusztig’s paper \cite{LusztigICC}: they provide the stratification with
respect to which character sheaves are constructible. Properties of such objects and of
their closures were studied in \cite{CE1} to describe sheets for the conjugacy action
of a reductive group $G$ on itself.

In this work we define a group analogue of Losev's birational sheets.
Motivation behind the study of this problem
is its connection with Representation Theory; in particular, we are interested in the aforementioned Losev's conjecture: a group analogue will be object of study in a forthcoming project.
After introducing some notation, in Section \ref{s_birind} we collect and reorganize existing results on induction of conjugacy classes. Induction of unipotent classes was defined by Lusztig and Spaltenstein in \cite{LS79} and it was then generalized to a non-unipotent conjugacy class in \cite{CE1} readapting arguments of \cite{Borho} to the group case.
Following this approach and inspired by \cite{Lo2016}, we define birational induction and weakly birational induction of a conjugacy class requiring birationality of two related maps and we compare these two notions.
The last part of Section \ref{s_birind} is devoted to extending properties enjoyed by induction to the case of birational induction.
In particular, Lemma \ref{lem_suffbir} states a criterion which gives a sufficient condition for a unipotent conjugacy class to be birationally induced.
The first main result of the work is Theorem \ref{thm_unibirind} in Section \ref{s_unique}, where we prove that any conjugacy class of $G$ can be  weakly birationally induced in a unique way up to $G$-conjugacy under some minimality conditions on the data needed to define induction.
 Section \ref{s_jcbs} recalls some notions on Jordan classes in $G$.
We restrict to the case $G$ semisimple and simply connected to define and describe the birational closure of a Jordan class. 
Subsequently we define birational sheets: in Theorem \ref{thm_partition} we prove that they form a partition of the group. We proceed to compare birational sheets with sheets from a structural point of view. 
The section ends with some remarks on a possible generalization of the results for $G$ not necessarily simply connected, obtained by relaxing the requirements on the induction.
Finally, Section \ref{s_locgeom} analyzes the local geometry of the birational closure of a Jordan class under the assumption that $G$ is semisimple and simply connected, using results from \cite{ACE}.
As an application, for $G$ semisimple simply connected, we show that birational sheets are unibranch with smooth normalization and that birational sheets are smooth in the classical case (Theorems \ref{thm_unibranch} and \ref{thm_smooth}).

\section{Notations and conventions} \label{s_not}
Let $G$ be a complex connected reductive linear algebraic group and let $\mathfrak{g} \coloneqq \Lie(G)$ be its Lie algebra.
For an algebraic subgroup $K \leq G$, we denote by $K^\circ$ its identity component, by $Z(K)$ its centre, and by ${\rm Aut}(K)$ the set of its automorphisms as an algebraic group.
If $X$ is a $K$-set, $X/K$ denotes the set of $K$-orbits.
When $K \leq G$ acts regularly on a variety $X$ and $x \in X$, the $K$-orbit of $x$ is denoted by $K \cdot x$ or $\orb^X_x$. For any $n \in \N$, we define the locally closed subsets $X_{(n)} \coloneqq \{ x \in X \, \mid \, \dim K \cdot x = n\} \subset X$. A \emph{sheet} of $X$ for the action of $K$ is an irreducible component of $X_{(n)}$ for some $n \in \N$ such that $X_{(n)} \neq \varnothing$. For $Y \subset X$, the regular locus of $Y$ is $Y^{reg} = Y \cap X_{(\bar{n})}$, where $\bar{n} = \max \{n \in \N \, \mid \, Y \cap X_{(n)} \neq \varnothing \}$.
For $Y \subset X$, the normalizer of $Y$ in $K$ is
$N_K(Y) \coloneqq \{ k \in K \, \mid \, k \cdot y \in Y \mbox{ for all } y \in Y\}$.
For $x \in X$, its stabilizer is denoted $K_x \coloneqq \{ k \in K \, \mid \, k \cdot x = x\}$.
When we consider the conjugacy (resp. the adjoint) action of $G$ on itself (resp. on $\mathfrak{g}$) we adopt the following notation for stabilizers, very common in the literature.
When $G$ acts on $X = G$ via conjugation, for  $g \in G$, we write $C_G(g)\coloneqq G_g = \{ k \in G \, \mid \, kgk^{-1} = g \}$.
For $Y \subset G$, the centralizer of $Y$ in $G$ is $C_G (Y) \coloneqq \bigcap_{y \in Y} C_G(y)$. 
When $G$ acts on $X = \lieg$ via the adjoint action, for $\eta \in \mathfrak{g}$, we write $C_G(\eta) \coloneqq G_\eta = \{ h \in G \, \mid \, \Ad(h)(\eta) = \eta \}$.
For $g \in G$ we define  $\mathfrak{c_g}(g) \coloneqq \{\xi \in \mathfrak{g} \, \mid \, \Ad(g)(\xi) = \xi \}$. 
Similarly, for $\eta \in \mathfrak{g}$, we write $\mathfrak{c_g}(\eta) \coloneqq \mathfrak{g}_\eta  = \{ \xi \in \mathfrak{g} \, \mid \, [\eta, \xi] = 0 \}$.
We define the centralizer of a Lie subalgebra $\mathfrak{k} \subset \mathfrak{g}$ as $\mathfrak{c_g(k)} \coloneqq \{\xi \in \mathfrak{g} \mid [\eta,\xi] = 0 \mbox{ for all } \eta \in \mathfrak{k}\}$.

Let $T$ be a fixed maximal torus in $G$, $B$ a fixed Borel subgroup of $G$ containing $T$. We set $\mathfrak{g} \coloneqq \mathrm{Lie} ( G )$, $\mathfrak{h} \coloneqq \mathrm{Lie} ( T ) $, $\mathfrak{b} \coloneqq \mathrm{Lie} ( B )$. 
The symbol $P$ (resp. $\mathfrak{p}$) denotes a parabolic subgroup of $G$ (resp. a parabolic subalgebra of $\mathfrak{g}$).
A Levi subgroup $L \leq G$ (resp. a Levi subalgebra $\mathfrak{l} \subset \mathfrak{g}$) is a Levi factor of a parabolic $P \leq G$ (resp. $\mathfrak{p} \subset \mathfrak{g}$).

We denote by $W$ the Weyl group of $G$, by $\Phi$ the root system associated to $T$, by $\Phi^+$ the set of positive roots relative to $B$ and by $\Delta$ the base for $\Phi$ extracted from $\Phi^+$.
For each $\alpha \in \Phi$, we denote by $U_\alpha$ the corresponding root subgroup in $G$.
When $\Phi$ is irreducible, we write $\Delta =\{ \alpha_i \mid i = 1, 2, \dots, n \}$ following the numbering in \cite[Planches I–IX]{Bour} and $-\alpha_0$ for the highest root with respect to $\Delta$. We set $\tilde{\Delta} = \Delta \cup \{\alpha_0 \}$. 

A standard parabolic subgroup is $P \leq G$ such that $B \leq P$: then there is $\Theta \subset \Delta$ such that $P = P_\Theta \coloneqq L_\Theta U_\Theta$, where $L_\Theta \coloneqq \langle T, U_{\alpha}, U_{-\alpha} \,  \mid \, \alpha \in \Theta \rangle$ is called a \emph{standard Levi subgroup} and $U_\Theta = \prod_{\alpha \in \Phi^+ \setminus \Z \Phi} U_\alpha$ is the unipotent radical of $P_\Theta$.
Standard parabolic (resp. standard Levi) subalgebras are the Lie algebras of standard parabolic (resp. standard Levi) subgroups.

For $s \in G$ semisimple, $C_G(s)^\circ$ is called a \emph{pseudo-Levi subgroup}, following \cite{SommersBC}. 
If $\Phi$ is irreducible, any pseudo-Levi subgroup of $G$ is conjugate to a \emph{standard pseudo-Levi group} $M_\Theta \coloneqq \langle T, U_{\alpha}, U_{-\alpha} \, | \, \alpha \in \Theta \rangle$ for some $\Theta \subsetneq \tilde\Delta$, \cite[Proposition 3]{SommersBC}.
Let $M \leq G$ be a pseudo-Levi and let $Z = Z(M)$. For $z \in Z$, we say that the connected component $Z^\circ z \subset Z$ satisfies the regularity property (RP) for $M$ if
\begin{equation} \label{RP}
C_G(Z^\circ z)^\circ = M \tag{RP}.
\end{equation} 

If $K \leq G$ is connected reductive and $\mathfrak{k} \coloneqq \Lie(K)$, we write $\mathcal{U}_K$ 
for the unipotent variety of $K$ and $\mathcal{N}_{\mathfrak{k}}$ 
for the nilpotent cone of $\mathfrak{k}$; we also set $\mathcal{U} \coloneqq \mathcal{U}_G$ and $\mathcal{N} \coloneqq \mathcal{N}_{\mathfrak{g}}$. The set of all $K$-conjugacy classes of $K$ is denoted $K/K$.
A central isogeny $\pi \colon K \to \overline{K}$ is a surjective group homomorphism with $\ker \pi \leq Z(K)$.

When we write $g_1 = su$, $g_2 = rv \in G$ or $\xi = \sigma + \nu \in \mathfrak{g}$ we implicitly assume that $su$ (resp. $rv$, resp. $\sigma + \nu$) is the Jordan decomposition of $g_1$ (resp. $g_2$, resp. $\xi$), with $s, r$ semisimple and $u, v$ unipotent ($\sigma$ semisimple and $\nu$ nilpotent, resp.). By convention, the elements of $\mathfrak{z(g)}$ are semisimple so that, for $\zeta \in \mathfrak{z(g)}$ and $\xi = \sigma + \nu \in \mathfrak{g}$, the semisimple part of $\zeta + \xi$ is $\zeta + \sigma$.


In the examples, we will use the following conventions. For $n \in \N \setminus \{0\}$, let $\mathtt{J}'_n$ be the square matrix of order $n$ whose elements on the antidiagonal are $1$ and all other entries are $0$. We denote by
$\mathtt{J}_{2n} \coloneqq \left( \begin{smallmatrix} 0 & \mathtt{J}'_n \\ -\mathtt{J}'_n & 0 \end{smallmatrix} \right)$, and we realize the symplectic group as ${\rm Sp}_{2n}(\C) = \{A \in {\rm GL}_{2n}(\C) \mid A^T \mathtt{J}_{2n} A = \mathtt{J}_{2n} \}$. As fixed Borel we choose the subgroup of upper-triangular matrices in ${\rm Sp}_{2n}(\C)$ and as fixed torus we select the subgroup of diagonal matrices in ${\rm Sp}_{2n}(\C)$.

\section{Birational induction} \label{s_birind}
For a start, we recollect some facts from \cite{BK79, Jantzen2004, Lo2016}.
\subsection{Birationality of the generalized Springer map}
Let $H \leq G$ be closed and let $X$ be an irreducible $H$-variety, then $H$ acts on $G \times X$ via $h \cdot ( g, x ) = ( gh^{-1}, h \cdot x )$ for $h \in H$, $g \in G$, $x \in X$. The orbit set $(G \times X)/H$ is denoted by $G \times^H X$ and it is endowed with the structure of an irreducible variety of dimension $\dim G/H + \dim X$; we write $g * x$ for the class of $(g,x) \in G \times X$. For $g' \in G$ and $g * x \in G \times^H X$, we have a $G$-action on $G \times^H X$ via $g' \cdot (g*x) = g'g*x$.
There is a one-to-one correspondence between $H$-stable subsets of $X$ and $G$-stable subsets of $G \times^H X$ assigning the orbit $H \cdot x$ to the orbit $G \cdot (1 * x)$; we also have $G_{g * x} = g H_x g^{-1}$, for all $g \in G, x \in X$.

Suppose in addition $X$ is a subvariety of a $G$-stable variety $Y$  and that the $H$-action on $X$ is the restriction of the $G$-action on $Y$. 
Then we define a surjective $G$-equivariant morphism
$\gamma:  G \times^H X  \to  G \cdot X$ via $g*x  \mapsto  g \cdot x$.
For $x \in X$, we have (see \cite[Proof of Lemma 7.10]{BK79}):
\begin{align} \label{eq_card}
\lvert \gamma^{-1}(x) \rvert = [G_x : H_x] \left\lvert \dfrac{G \cdot x \cap X}{H} \right\rvert.
\end{align}

\begin{lemma} \label{lem_jantzen}
Let $P \leq G$ be parabolic and let $X$ be a closed $P$-subvariety of the $G$-variety $Y$.
Let $\gamma \colon G \times^P X \to G \cdot X$ be the map $g * x \mapsto g \cdot x$.
Assume:
\begin{gather}
G \cdot X = \overline{\orb}  \mbox{ for some } G \mbox{-orbit } \orb \subset Y \tag{H1} \label{h_1}\\
\dim G \cdot X = \dim G/P + \dim X \tag{H2} \label{h_2}. 
\end{gather}
Then:
\begin{enumerate}[noitemsep, nolistsep]
\item[(i)]$\orb \cap X$ is a single $P$-orbit;
\item[(ii)]$\widetilde{\orb} \coloneqq \gamma^{-1}(\orb)$ is a single $G$-orbit, open and dense in $G \times^P X$;
\item[(iii)] There is a natural isomorphism $G / P_x \xrightarrow{\sim} \widetilde{\orb}$ and $\gamma$ restricts to an unramified covering $\widetilde{\orb} \to \orb$ of degree $[G_x : P_x]$.
\end{enumerate}
\end{lemma}
\begin{proof}
We want to make use of \cite[Lemma 8.8]{Jantzen2004}. Observe first that $G \times^P X$ identifies with a locally closed subvariety of $G/P \times Y$ via the closed embedding $g*x \mapsto (gP, g \cdot x)$ (see \cite[§7.9]{BK79}).
Jantzen's proof still holds substituting a $P$-submodule of a $G$-module with the closed $P$-stable variety $X$ inside the $G$-stable variety $Y$.
Moreover, if $\varphi_x \colon G \to \orb^G_x$ is the orbit map, the condition $\Lie(G_x) = \ker ({\rm d}_1 \varphi_x)$ is always fulfilled when the base field is $\C$ (see \cite[Remark on Lemma 8.8]{Jantzen2004}).
For these reasons, to prove our result, we can proceed as in \cite[Lemma 8.8]{Jantzen2004}, provided that we show that its assumptions are satisfied, namely that 
there exists $x \in X$ with $\overline{P \cdot x} = X$ and $G^\circ_x \leq P$.
By \cite[Theorem 5.1.6]{Spr}, there exists a non-empty open $U \subset G \cdot X$ whose fibres through $\gamma$ are finite: indeed \eqref{h_2} implies $\dim \gamma^{-1}(x) = \dim G \times^P X - \dim U = 0$ for all $x \in U$.
Moreover, $\orb$ is open in $\overline{\orb} = G \cdot X$ by \eqref{h_1} so that $U \cap \orb \neq \varnothing$.
By $G$-equivariance of $\gamma$, we have that $\gamma^{-1}(x)$ is finite for all $x \in \orb$.
Again by \eqref{h_1}, the set $X$ meets $\orb$ non-trivially.
By \eqref{eq_card}, for $x \in \orb \cap X$, we have $G_x^\circ \leq P$ and $\orb \cap X$ is a union of finitely many $P$-orbits. 
Since $X$ is irreducible and $\orb \cap X$ is open in $\overline{\orb} \cap X = X$ by \eqref{h_1}, we have $X = \overline{\orb \cap X} = \overline{P \cdot x}$, for some $x \in \orb \cap X$.
\end{proof}

Let $P = LU$ be the Levi decomposition of $P$.
Let $\orb^L \in L/L$ and specialize the above construction to the case $X = \overline{\orb^L}U$, $Y = G$, where $P$ acts on $X$ by conjugacy.
The \emph{generalized Springer map} is:
\begin{align} \label{gsp}
\gamma \colon  G \times^P \overline{\orb^L}U & \to  G \cdot (\overline{\orb^L}U) \tag{GSM} \\
g*x & \mapsto  g x g^{-1} \nonumber
\end{align}
The image of $\gamma$ is the closure of a single conjugacy class $\orb^G \in G/G$, called the conjugacy class \emph{induced} from $(L, \orb^L)$, see \cite{LS79, CE1}.
The definition of induced conjugacy class only depends on the Levi $L \leq G$ and on the class $\orb^L$ and not on the parabolic $P$ containing $L$.
For any Levi subgroup $L \leq G$, we introduce notation $\mathrm{Ind}_L^G \colon L/L \to G/G$ associating to each class $\orb^L \in L/L$ the induced conjugacy class $\mathrm{Ind}_L^G \orb^L \in G/G$.
We have $\codim_G (\mathrm{Ind}_L^G \orb^L) = \codim_L \orb^L$ (see \cite[Proposition 4.6]{CE1}). 
A unipotent conjugacy class $\orb^G$ is \emph{rigid} in $G$ if it cannot be induced by a proper Levi subgroup $L \lneq G$ and a unipotent class $\orb^L \in \mathcal{U}_L/L$.

The case of birational induction of adjoint orbits has been studied in \cite{broer_1998, Nami2009, Fu2010, Lo2016}. Here we deal with birationality of $\gamma$ for induction of conjugacy classes.

\begin{lemma}[Birationality criterion] \label{lem_bir_crit}
Let $P \leq G$ be a parabolic subgroup with Levi decomposition $P=LU$, let $\orb^L \in L/L$,  let $\orb^G = \ind_L^G \orb^L$, let $\gamma$ as in \eqref{gsp}, and $\widetilde{\orb} \coloneqq \gamma^{-1}(\orb^G)$. The following are equivalent:
\begin{enumerate}[noitemsep,nolistsep]
\item[(i)] $\gamma$ is birational;
\item[(ii)] $\gamma$ maps $\widetilde{\orb}$ isomorphically to $\orb^G$;
\item[(iii)] for all $x \in  \orb^G \cap \overline{\orb^L} U$, we have $C_G(x) = C_P(x)$;
\item[(iv)] there exists $x \in  \orb^G \cap \overline{\orb^L} U$ such that $C_G(x) = C_P(x)$.
\end{enumerate}
\end{lemma}

\begin{proof}
We want to apply Lemma \ref{lem_jantzen} with $P$ acting by conjugacy on $X = \overline{\orb^L}U$ and $\orb = \orb^G \in G/G$.
Remark that $\dim G \cdot (\overline{\orb^L}U) = \dim \orb^G = \dim G - \dim L + \dim \orb^L = \dim G/P + \dim \orb^L + \dim U = \dim G/P + \dim (\overline{\orb^L}U)$, hence both \eqref{h_1}, \eqref{h_2} are satisfied.
 
Let (i) hold. By Lemma \ref{lem_jantzen}, $\gamma$ induces a finite covering $\widetilde{\orb} \to \orb^G$ and these sets are open and dense, this covering is injective by birationality and $G$-equivariance, i.e. $\gamma(\widetilde{\orb}) \simeq \orb^G$. The converse implication is trivial.
(ii) holds if and only if the degree of $\gamma$ over $\orb^G$ is $1$. By \eqref{eq_card} this is equivalent to $[C_G(x):C_P(x)] = 1$ for $x \in \orb^G \cap \overline{\orb^L} U $, that is (iii).
Finally (iii) obviously implies (iv) and since $\orb^G \cap \overline{\orb^L}U$ is a single $P$-orbit, (iv) implies (iii) by $P$-equivariance.
\end{proof}

\subsection{Reduction to the unipotent case}
We will make use of the following result:
\begin{lemma} \label{lem_itsalevi}
Let $M \leq G$ be a pseudo-Levi subgroup, $z \in Z \coloneqq Z(M)$. Then $Z^\circ z$ satisfies \eqref{RP} for $M$ if and only if $M$ is a Levi subgroup of $C_G(z)^\circ$.

In particular, if $M = C_G(s)^\circ$ for a semisimple element $s \in G$, then $Z(M)^\circ s$ satisfies \eqref{RP} for $M$.
\end{lemma}

\begin{proof}
We prove first that $C_G(Z^\circ z)^\circ = C_G(z)^\circ \cap C_G(Z^\circ) = C_{C_G(z)^\circ}(Z^\circ)$. 
The chain of inclusions $C_G(Z^\circ z)^\circ \leq C_G(z)^\circ \cap C_G(Z^\circ) \leq C_G(Z^\circ z)$ is trivial and $C_G(z)^\circ \cap C_G(Z^\circ) = C_{C_G(z)^\circ}(Z^\circ)$ is connected by \cite[Corollary 11.12]{borel}, hence $C_G(z)^\circ \cap C_G(Z^\circ) \leq C_G(Z^\circ z)^\circ$.

Assume $Z^\circ z$ satisfies \eqref{RP} for $M$, then $M = C_G(Z^\circ z)^\circ =  C_{C_G(z)^\circ}(Z^\circ)$, which is a Levi subgroup of $C_G(z)^\circ$ by\cite[Corollary 20.4]{borel}.
Conversely, if $M$ is a Levi in $C_G(z)^\circ$, then by \cite[Proposition 14.18]{borel}, $M = C_{C_G(z)^\circ}(Z(M)^\circ) = C_{C_G(z)^\circ}(Z^\circ) = C_G(Z^\circ z)^\circ$.

The last statement follows directly.
\end{proof}

\begin{remark} \label{rk_itsalevi}
Let $M \leq G$ be a pseudo-Levi subgroup, $z \in Z \coloneqq Z(M)$. Then $Z^\circ z$ satisfies \eqref{RP} if and only if $Z^{reg} \cap Z^\circ z \neq \varnothing$, see \cite[Remark 3.6]{CE1}.
\end{remark}

In the remainder of this section, we recall how induction of a conjugacy class of $G$ is related to induction of a unipotent class in a pseudo-Levi subgroup of $G$.

Let $L \leq G$ be a Levi subgroup and let $P = LU \leq G$ be a parabolic with Levi factor $L$ and let $su \in L$.
It was proven in \cite[Proposition 4.6]{CE1} that:
\begin{equation} \label{induzione}
  \mathrm{Ind}_L^G  \orb^L_{su} =   G \cdot ( s \mathrm{Ind}_{C_L(s)^\circ}^{C_G(s)^\circ} \orb^{C_L(s)^\circ}_u ).
\end{equation} 
Notice that $C_L(s)^\circ$ is a Levi subgroup in $C_G(s)^\circ$ (see \cite[Proof of Proposition 4.6]{CE1}).

Induction is transitive, i.e. if $M \leq G$ is a Levi subgroup, $L$ is a Levi subgroup of $M$  and $\orb^L_{su} \in L/L$, then:
\begin{align*}
\ind_M^G ( \ind_L^M \orb^L_{su} ) &= \ind_M^G(M \cdot (s \ind_{C_L(s)^\circ}^{C_M(s)^\circ} \orb^{C_L(s)^\circ}_u)) =\\
&= G \cdot (s \ind_{C_M(s)^\circ}^{C_G(s)^\circ}( \ind_{C_L(s)^\circ}^{C_M(s)^\circ} \orb^{C_L(s)^\circ}_u)) =\\
&= G \cdot (s \ind_{C_L(s)^\circ}^{C_G(s)^\circ} \orb^{C_L(s)^\circ}_u) = \ind_M^G \orb^L_{su},
\end{align*}
where we used \cite[§1.7]{LS79}. 

We can assume that $T \leq L$ and $s \in T$. If $P \geq B$, we have that $C_B(s)^\circ = C_B(s)$ is a Borel subgroup of $C_G(s)^\circ$ and $C_P(s)^\circ$ is a parabolic subgroup of $C_G(s)^\circ$. The equality $C_P(s)^\circ = P \cap C_G(s)^\circ$ holds and $C_L(s)^\circ$ is a Levi factor of $C_P(s)^\circ$; we write $ U_{C_P(s)^\circ}$ for the unipotent radical of $C_P(s)^\circ$.
We compare the two morphisms:
\begin{align}
\gamma \colon G \times^P \overline{\orb^L_{su}}U  &\to \overline{\mathrm{Ind}_L^G \orb^L_{su}} \label{ggsp} \\
\gamma_s \colon C_G(s)^\circ \times^{C_P(s)^\circ} \overline{\orb^{C_L(s)^\circ}_u} U_{C_P(s)^\circ}  &\to \overline{\mathrm{Ind}_{C_L(s)^\circ}^{C_G(s)^\circ} \orb^{C_L(s)^\circ}_u} \label{ugsp}
\end{align}

\begin{lemma} \label{lem_gsp_conn}
Let $\gamma$ and $\gamma_s$ be as in \eqref{ggsp} and \eqref{ugsp}, respectively. Set $\orb^G \coloneqq \mathrm{Ind}_L^G \orb^L_{su}$ and $\orb^{C_G(s)^\circ} \coloneqq \mathrm{Ind}_{C_L(s)^\circ}^{C_G(s)^\circ} \orb^{C_L(s)^\circ}_u$. Then:
\begin{enumerate}[noitemsep, nolistsep]
\item[(i)] Birationality of $\gamma$ implies birationality of $\gamma_s$.
\item[(ii)] Suppose in addition that $C_G(s)=C_G(s)^\circ$. If $\gamma_s$ is birational, then $\gamma$ is birational. In particular, if $G$ is semisimple simply connected, this is always the case.
\end{enumerate}
\end{lemma}

\begin{proof}
We will prove birationality verifying condition (iv) of Lemma \ref{lem_bir_crit}.

(i) Suppose $\gamma$ is birational.
Let $v \in \orb^{C_L(s)^\circ}_u U_{C_P(s)^\circ} \cap \orb^{C_G(s)^\circ}$.
Then $sv \in  \orb^G = G \cdot (s \orb^{C_G(s)^\circ})$ by \eqref{induzione} and $sv \in s\orb^{C_L(s)^\circ}_u U_{C_P(s)^\circ} \subset \orb^L_{su} U$, so that $sv \in \orb^L_{su} U \cap \orb^G$.
Since $v \in C_G(s)^\circ$,  the unipotent part of $sv$ is $v$.
Birationality of $\gamma$ yields $C_G(sv) \leq P$, so $C_{C_G(s)^\circ}(v) \leq C_{C_G(s)}(v) \leq P$.
Now $C_{C_G(s)^\circ}(v) \leq C_G(s)^\circ$, hence $C_{C_G(s)^\circ}(v) \leq P \cap C_G(s)^\circ = C_P(s)^\circ$ and  $\gamma_s$ is birational.

(ii) Assume that $C_G(s)$ is connected and that  $\gamma_s$ is birational.
Choose an element $sv \in s\orb^{C_G(s)^\circ} \cap s\orb^{C_L(s)^\circ} U_{C_P(s)^\circ} \subset \orb^G \cap \orb^L_{su}$. Then $C_G(sv) = C_{C_G(s)^\circ}(v) \leq {C_P(s)^\circ} \leq P$.
\end{proof}

The following example shows that in general, the birationality of $\gamma_s$ does not imply the birationality of $\gamma$.
\begin{example}
Let $\widetilde{G} = {\rm Sp}_4(\C)$ and let $G = {\rm PSp}_4(\C) = \widetilde{G}/Z(\widetilde{G})$, we denote its elements by matrices in square brackets.
Let $s=
\left[\begin{smallmatrix}
-1&0&0&0\\
0&1&0&0\\
0&0&1&0\\
0&0&0&-1\end{smallmatrix}\right] \in G$.
Let $P \coloneqq P_{\Theta}$ be the standard parabolic 
with $\Theta = \{\alpha_2\}$,  unipotent radical $U$ and Levi factor $L  \coloneqq L_{\Theta}$.
Remark that $C_G(s)^\circ \lneq C_G(s)$ and that $C_L(s)^\circ = L$.
Let $u$ be in the regular unipotent class of $L$.
Then $\ind_L^{C_G(s)^\circ} (\orb^L_{u}) = \orb^{C_G(s)^\circ}_v$ with $v 
=\left[\begin{smallmatrix}
1&0&0&-1\\
0&1&1&0\\
0&0&1&0\\
0&0&0&1\end{smallmatrix}\right]$ while $\ind_L^G \orb^L_{su} = \orb^G_{sv}$.
We have that $\gamma_s$ as in \eqref{ugsp} is birational, since $C_{C_G(s)^\circ}(v) \leq C_P(s)^\circ$. On the other hand, $\gamma$  as in \eqref{ggsp} is not birational, because $C_G(sv) \not\leq P$, for example
$\left[\begin{smallmatrix}
0&1&0&0\\
1&0&0&0\\
0&0&0&1\\
0&0&1&0\end{smallmatrix}\right] \in C_G(sv) \setminus P$.
\end{example}

\subsubsection{Induction of adjoint orbits in a Lie algebra} \label{sss_lie_uguale}
In this section we describe birational induction in the case of the adjoint action of $G$ on its Lie algebra $\mathfrak{g}$.
We follow the approach of \cite[\S 1.4 and \S 4]{Lo2016}: we alert the reader that in \cite{Lo2016} most results are formulated in terms of $G$-coadjoint orbits of $\mathfrak{g}^*$.
It is possible to translate all the statements in terms of $G$-adjoint orbits of $\mathfrak{g}$ after choosing a $G$-equivariant non-degenerate symmetric associative bilinear form on $\mathfrak{g}$ (see \cite[\S 4, Proposition 5]{bourba1}),  which yields a $G$-equivariant isomorphism of vector spaces $\mathfrak{g} \simeq \mathfrak{g}^*$.

For a Levi subalgebra $\mathfrak{l} \subset \mathfrak{g}$, let $\zeta \in \mathfrak{z(l)}$ and $\orb^\mathfrak{l} \in \mathcal{N}_\mathfrak{l}/L$\footnote{We remark that induction in $\mathfrak{g}$ can be defined for the adjoint orbit of any element $\sigma + \nu$ in a Levi subalgebra $\mathfrak{l} \subset \mathfrak{g}$, see \cite[\S 2]{Borho}. The reduction to the adopted definition can be obtained from \cite[Satz 2.1, 3. Fall]{Borho}.}.
Include $\mathfrak{l}$ as the Levi factor of a parabolic subalgebra $\mathfrak{p} = \mathfrak{l} + \mathfrak{n}$ and let $P \leq G$ be such that $\Lie(P) = \mathfrak{p}$.
The group $P$ acts via the adjoint action on $(\zeta + \overline{\orb^\mathfrak{l}} + \mathfrak{n})$ and we have the generalized Springer map:
\begin{align} \label{gsm_losev}
\gamma \colon  G \times^P (\zeta + \overline{\orb^\mathfrak{l}} + \mathfrak{n}) & \to  \Ad(G)(\zeta + \overline{\orb^\mathfrak{l}} + \mathfrak{n}) \\
g*\eta & \mapsto  \Ad(g)(\eta) \nonumber
\end{align}
Then ${\rm im} (\gamma) = \overline{\orb^\mathfrak{g}}$ for a unique $\orb^\mathfrak{g} \in \mathfrak{g}/G$,
called the orbit \emph{induced from the induction data} $(\mathfrak{l}, \zeta, \orb^\mathfrak{l})$ and denoted $\ind_\mathfrak{l}^\mathfrak{g} (\zeta + \orb^\mathfrak{l})$.
The orbit $\orb^\mathfrak{g} = \ind_\mathfrak{l}^\mathfrak{g} (\zeta + \orb^\mathfrak{l})$ is nilpotent if and only if $\zeta = 0$: in this case we write $\orb^\mathfrak{g}$ is induced from $(\mathfrak{l}, \orb^\mathfrak{l})$.

When $\gamma$ is birational, $\ind_\mathfrak{l}^\mathfrak{g} (\zeta + \orb^\mathfrak{l})$ is said to be \emph{birationally induced} from $(\mathfrak{l}, \zeta, \orb^\mathfrak{l})$. A nilpotent orbit $\orb \subset \mathfrak{g}$ is said to be \emph{(birationally) rigid} if it cannot be (birationally) induced from a {\em proper} Levi subalgebra $\mathfrak{l} \subsetneq \mathfrak{g}$.

We have $\codim_\mathfrak{g} \orb^{\mathfrak{g}} = \codim_\mathfrak{l} (\zeta + \orb^{\mathfrak{l}})$, see \cite[Satz 3.3]{Borho}. Therefore, the hypotheses of Lemma \ref{lem_jantzen} are satisfied with $X = \zeta + \overline{\orb^\mathfrak{l}} + \mathfrak{n}$ and $Y = \mathfrak{g}$. Moreover, an analogue of Lemma \ref{lem_bir_crit} holds.

Since $\zeta \in \lieg$ is semisimple, $C_G(\zeta)$ is connected, see \cite[§3]{Steinberg75}.
In the setting of \eqref{gsm_losev}, we have $\Lie(C_G(\zeta)) = \mathfrak{c_g}(\zeta)$ and $C_P(\zeta) = P \cap C_G(\zeta)$ is a parabolic subgroup of $C_G(\zeta)$; moreover, $\mathfrak{l}$ is a Levi factor of the parabolic subalgebra $\mathfrak{p}_\zeta \coloneqq \Lie(C_P(\zeta))$, write $\mathfrak{n}_\zeta$ for its nilradical. Consider the generalized Springer map:
\begin{equation} \label{gsm_nilp}
\gamma_\zeta \colon  C_G(\zeta) \times^{C_P(\zeta)} (\overline{\orb^{\mathfrak{l}}} + \mathfrak{n}_\zeta)  \to  \Ad(C_G(\zeta))(\overline{\orb^\mathfrak{l}} + \mathfrak{n}_\zeta) 
\end{equation}

\begin{remark} \label{rk_lie}
The orbit $\orb^\mathfrak{g}$ is birationally induced from $(\mathfrak{l}, \zeta, \orb^\mathfrak{l})$ if and only if the nilpotent orbit
$\ind_\mathfrak{l}^{\mathfrak{c_g}(\zeta)} \orb^\mathfrak{l} \in \mathfrak{c_g}(\zeta)/C_G(\zeta)$ is birationally induced from $(\mathfrak{l}, \orb^\mathfrak{l})$.
\end{remark}

\begin{proof}
Lemma \ref{lem_gsp_conn} still holds with the necessary adjustments, so that $\gamma$ in \eqref{gsm_losev} is birational if and only if $\gamma_\zeta$ in \eqref{gsm_nilp} is birational.
\end{proof}

%

\subsection{Birationality for induction of conjugacy classes}
In this section we discuss the definition of birational induction of a class $\orb^G_{sv} \in G/G$. 

\begin{definition} \label{defn_birind}
Let $su \in L$ and let $\orb^G = \mathrm{Ind}_L^G \orb^L_{su}$.  We say that $\orb^G$ is:
\begin{enumerate}[noitemsep, nolistsep]
\item[(a)] \emph{birationally induced} from $(L, \orb^L_{su})$ if the generalized Springer map
\begin{align*}
\gamma \colon G \times^P \overline{\orb^L_{su}}U  &\to \overline{\mathrm{Ind}_L^G \orb^L_{su}}
\end{align*}
defined in \eqref{ggsp} is birational;
\item[(b)] \emph{weakly birationally induced} from $(L, \orb^L_{su})$ if the generalized Springer map
\begin{align*}
\gamma_s: C_G(s)^\circ \times^{C_P(s)^\circ} \overline{\orb^{C_L(s)^\circ}_u} U_{C_P(s)^\circ}  &\to \overline{\mathrm{Ind}_{C_L(s)^\circ}^{C_G(s)^\circ} \orb^{C_L(s)^\circ}_u} 
\end{align*} defined in \eqref{ugsp} is birational.
\end{enumerate}
\end{definition}

\begin{remark} \label{rk_coincide}
It is a consequence of Lemma \ref{lem_gsp_conn} that if $\orb^G_{sv} \in G/G$ is birationally induced from $(L, \orb^L_{su})$ then it is also weakly birationally induced from $(L, \orb^L_{su})$.
Moreover, the two notions coincide when $G$ is semisimple simply connected or when $\orb^G = z\orb^G_u$ for $z \in Z(G)$ and $u \in \mathcal{U}$ (this is the case if and only if the inducing orbit is $z \orb^L$ with $\orb^L \in \mathcal{U}_L/L$).
For this reason, in such cases, we will always omit the adverb ``weakly''.
\end{remark}

We may drop one, or both of the elements of the pair of inducing data $(L, \orb^L_{su})$ in the notation when they are clear from the context or they are not relevant. In particular, we will say that the class $\orb^G \in G/G$ is (non-trivially) birationally induced (resp. weakly birationally induced) if there exists a proper Levi subgroup $L \lneq G$ and a conjugacy class $\orb^L \in L/L$ such that $\orb^G$ is birationally induced (resp. weakly birationally induced) from $(L, \orb^L)$.

In the following, we focus on  induction of unipotent classes and we show that most properties carry over to the birational case.

\subsubsection{Interaction with isogeny}
Induction and birational induction behave well with respect to Springer's isomorphism $\phi: \mathcal{N} \xrightarrow{\sim} \mathcal{U}$. We recall that $G$-equaviariance of $\phi$ yields $C_G(\nu) = C_G(\phi(\nu))$ for any $\nu \in \mathcal{N}$; in particular, the homogeneous spaces $\orb^\mathfrak{g}_\nu$ and $\orb^G_{\phi(\nu)}$ are isomorphic varieties, see \cite[Theorems 1.1, 1.2]{LiebeckSeitz}.

\begin{remark} \label{rk_ind_isogeny}
Let $\pi \colon G \to \overline{G}$ be an isogeny. Fix a Levi subgroup $L \leq G$ and let $\mathfrak{l} \coloneqq \Lie(L)$, set $\pi(L) = \overline{L}$.
Let $\phi_L \colon \mathcal{N}_\mathfrak{l} \xrightarrow{\sim} \mathcal{U}_L$ be Springer's isomorphism for $L$ and $\phi \colon \mathcal{N} \xrightarrow{\sim} \mathcal{U}$ Springer's isomorphism for $G$.
Let $\orb^{\mathfrak{l}} \in \mathcal{N}_\mathfrak{l}/L$.
\begin{enumerate}[noitemsep,nolistsep]
\item[(i)] $\phi(\ind_{\mathfrak{l}}^\mathfrak{g}{\orb^{\mathfrak{l}}}) = \ind_L^G (\phi_L(\orb^{\mathfrak{l}}))$. In particular, rigid 
nilpotent orbits in $\mathfrak{g}$ correspond to rigid unipotent classes in $G$.
\item[(ii)] If $\orb^L \in \mathcal{U}_L$, then $\pi \left( \ind_L^G \orb^L \right) = \ind_{\overline{L}}^{\overline{G}} (\pi(\orb^L))$, because Springer's isomorphism does not depend on the isogeny class of $G$.
\item[(iii)] A nilpotent orbit $\orb^\mathfrak{g} \in \mathcal{N}/G$ is birationally induced from $(\mathfrak{l}, \orb^{\mathfrak{l}})$ if and only if $\phi(\orb^\mathfrak{g})$ is birationally induced from $(L, \phi_L (\orb^{\mathfrak{l}}) )$, as Springer's isomorphism preserves centralizers.
\item[(iv)] Since (iii) is true independently of the isogeny class of $G$, the class $\orb^G \in \mathcal{U}/G$ is birationally induced from $(L, \orb^L)$ if and only if  $\pi(\orb^G)$ is birationally induced from $(\overline{L}, \pi(\orb^L))$.
\end{enumerate}
\end{remark}

\subsubsection{Independence of the choice of a parabolic subgroup} \label{sss_indeparab}
In \cite[§4]{Lo2016} it is proven that birationality of induction of a nilpotent orbit is independent of the parabolic subgroup. The result therein is a consequence of deformation theory. Here we prove that birationality of induction of a unipotent class $\orb^G$ only depends on the $G$-conjugacy class of the inducing pair $(L, \orb^L)$. Our proof is obtained  independently of Losev's analogue result.
We recall a notion that will be needed in the proof:  if $f \colon X \to Y$ is a dominant rational map of varieties, the degree of $f$ is defined as the degree of the function fields extension $\deg f \coloneqq [\C(X) \colon \C(Y)]$.
Moreover, if $f$ is a finite map and $y \in Y$ is a generic point, then $\left\vert f^{-1}(y) \right\vert = \deg f$, see \cite[Proposition 7.16]{Harris}.

\begin{lemma}
Consider two parabolic subgroups $P, Q \leq G$ with Levi decompositions $P = LU_P, Q =  LU_Q$ respectively. Let $\orb^L \in \mathcal{U}_L/L$ and set $\orb^G = \ind_L^G(\orb^L)$.
Then the generalized Springer map $\gamma_P \colon G \times^P (\overline{\orb^L} U_P) \to \overline{\orb^G}$ is birational if and only if the generalized Springer map $\gamma_Q \colon G \times^Q (\overline{\orb^L} U_Q) \to \overline{\orb^G}$ is birational.
\end{lemma}

\begin{proof}
Set $\mathfrak{p} \coloneqq \Lie(P)$, $\mathfrak{q} \coloneqq \Lie(Q)$ with Levi decompositions $\mathfrak{p} = \mathfrak{l} + \mathfrak{n_p}$ and $\mathfrak{q} = \mathfrak{l} + \mathfrak{n_q}$, respectively.
Let $\orb^\mathfrak{l} \in \mathcal{N}_\mathfrak{l}/L$ (resp. $\orb^\mathfrak{g} \in \mathcal{N}/G$) such that $\phi_L(\orb^\mathfrak{l}) = \orb^L$ (resp. $\phi(\orb^\mathfrak{g}) = \orb^G$). Consider the generalized Springer maps $
\gamma_\mathfrak{p} \colon G \times^P (\overline{\orb^\mathfrak{l}} + \mathfrak{n_p}) \to \overline{\orb^\mathfrak{g}}$ and $\gamma_\mathfrak{q} \colon G \times^Q (\overline{\orb^\mathfrak{l}} + \mathfrak{n_q}) \to \overline{\orb^\mathfrak{g}}$.
By Remark \ref{rk_ind_isogeny}, $\gamma_P$ (resp. $\gamma_Q$) is birational if and only if $\gamma_\mathfrak{p}$ (resp. $\gamma_\mathfrak{q}$) is birational. 
The degrees of $\gamma_\mathfrak{p}$ and of $\gamma_\mathfrak{q}$ are the same.
This follows from \cite[Proof of Corollary 3.9]{BorhoMacPh} where a formula for the degree of the Springer generalized map $\gamma_\mathfrak{p}$ is given in terms of $(\mathfrak{l}, \orb^\mathfrak{l})$, and these data are independent of the parabolic.
\end{proof}

\begin{remark}
Let $\mathfrak{l}$ be a Levi subalgebra of $\mathfrak{g}$, let $\zeta \in \mathfrak{z(l)}$ and $\orb^\mathfrak{l} \in \mathcal{N}_\mathfrak{l}/L$.
Then, by \S \ref{sss_lie_uguale}, $\ind_\mathfrak{l}^\mathfrak{g} (\zeta + \orb^\mathfrak{l})$ is birationally induced if and only if $\ind_\mathfrak{l}^{\mathfrak{c_g}(\zeta)} \orb^\mathfrak{l}$ is birationally induced. 
As a result, birationality of the induction of any $\orb^\mathfrak{g} \in \mathfrak{g}/G$ is independent of the chosen parabolic.
\end{remark}

\begin{remark} \label{rk_conj_data_bir}
Let $\vartheta \in {\rm Aut}(G)$, let $L \leq G$ be a Levi subgroup and $\orb^L \in L/L$. Then $\vartheta$ preserves parabolic subgroups, their Levi decompositions, and unipotent varieties of Levi subgroups. Thus, $\vartheta (\ind_L^G \orb^L) = \ind_{\vartheta(L)}^G \vartheta(\orb^L)$. 
By Lemma \ref{lem_bir_crit}, birationality of unipotent induction is preserved under action by $\mathrm{Aut}(G)$ on the inducing data because $\vartheta$ maps centralizers to centralizers. In particular, for  $g \in G$ and $\orb^L_u \in \mathcal{U}_L/L$, the class $\orb^G \coloneqq \ind_L^G(\orb^L_u)$ is birationally induced from $(L, \orb^L_u)$ if and only if $\orb^G$ is birationally induced from $(gLg^{-1}, \orb^{gLg^{-1}}_{gug^{}-1})$.
\end{remark}

\subsubsection{A sufficient condition for birationality}
Next result can be used to test if a unipotent class is birationally induced.
\begin{lemma} \label{lem_suffbir}
Let $\phi: \mathcal{N} \to \mathcal{U}$ be Springer's isomorphism and let $\overline{G} = G/Z(G)$. Let $\nu \in \mathcal{N}$ and let $\orb \coloneqq \orb^{\mathfrak{g}}_\nu$. If $C_{\overline{G}}(\nu)/C_{\overline{G}}(\nu)^\circ = \{1\}$, then:
\begin{enumerate}[noitemsep, nolistsep]
\item[(i)] if $\orb = \ind_\mathfrak{l}^\mathfrak{g} \orb^\mathfrak{l}$ for a Levi subalgebra $\mathfrak{l} \subset \mathfrak{g}$ and $\orb^\mathfrak{l} \in \mathcal{N}_\mathfrak{l}/L$, then $\orb$ is birationally induced from $(\mathfrak{l}, \orb^\mathfrak{l})$;
\item[(ii)] if $\phi(\orb) = \ind_L^G \orb^L$ for a Levi subgroup $L \leq G$ and $\orb^L \in \mathcal{U}_L/L$, then $\phi(\orb)$ is birationally induced from $(L, \orb^L)$.
\end{enumerate}
\end{lemma}

\begin{proof}
Let $P \leq G$ be parabolic with Levi decomposition $P=LU$ and let $\orb^\mathfrak{l} \in \mathcal{N}_\mathfrak{l}/L$ such that $\ind_\mathfrak{l}^\mathfrak{g} \orb^\mathfrak{l} = \orb$. If $C_{\overline{G}}(\nu) = C_{\overline{G}}(\nu)^\circ$, then $C_G(\nu) = C_P(\nu)$, hence (i) follows from Lemma \ref{lem_bir_crit}. Statement (ii) is a consequence of Remark \ref{rk_ind_isogeny}.
\end{proof}

\begin{example} Let $G$ be simple of type $\mathsf{C}$ and let $u \in \orb^G_{subreg}$, the subregular orbit in $\mathcal{U}$. Let $\Theta_{1} = \{\alpha_1\}$ and $\Theta_n = \{\alpha_n\}$.
Set $L_i \coloneqq L_{\Theta_i}$, then $\orb^G_{subreg}$ is induced both by $(L_1, \{1\})$ and   $(L_n, \{1\})$.
A direct computation of $C_G(u)$ shows that $\orb^G_{subreg}$ is birationally induced from $(L_1, \{1\})$, whereas the induction $\ind_{L_n}^G \{1\}$ is not birational.
In fact, $C_{\overline{G}}(u) / C_{\overline{G}}(u)^\circ \simeq \Z / 2 \Z$.
\end{example}

If $G$ is simple adjoint and $\nu \in \mathcal{N}$, the groups $C_G(\nu) / C_G(\nu)^\circ$ are known, see \cite[\S 6.1, \S 8.4]{CollMcGov}.

\subsubsection{Transitivity of birational induction}
For a Levi subgroup $M \leq G$, let $L$ be a Levi subgroup of $M$ and let $\orb^L \in \mathcal{U}_L/L$. We want to prove that $\ind_L^G \orb^L$ is birationally induced from $(L, \orb^L)$ if and only if $\ind_L^M \orb^L$ is birationally induced from $(L, \orb^L)$ and $\ind_M^G (\ind_L^M \orb^L)$ is birationally induced from $(M, \ind_L^M \orb^L)$.

We can work with a standard parabolic subgroup $P$ with Levi decomposition $P = LU$, where $L \leq G$ is a standard Levi.
Let $Q \geq P$ be parabolic with Levi decomposition $Q = MV$ where $M \leq G$ is a standard Levi subgroup. We have $L \leq M$ and $U \geq V$.
Moreover, $P \cap M$ is a parabolic subgroup of $M$ (standard with the choice $M \cap B$ for the Borel of $M$), with Levi decomposition of $P \cap M = L (U \cap M)$ and $U = (U \cap M) V$.
Let $\orb^M \coloneqq \ind_L^M \orb^L$ and $\orb^G \coloneqq \ind_L^G \orb^L = \ind_M^G \orb^M$ with the corresponding generalized Springer maps:
\begin{align*}
\gamma_L^G: G \times^P \overline{\orb^L}U \twoheadrightarrow \overline{\orb^G}, \\
\gamma_L^M: M \times^{P \cap M} \overline{\orb^L}(U \cap M) \twoheadrightarrow \overline{\orb^M}, \\
\gamma_M^G: G \times^Q \overline{\orb^M} V  \twoheadrightarrow \overline{\orb^G}.
\end{align*}

\begin{proposition}[Transitivity of birational induction] \label{prop_trans}
The map $\gamma_L^G$ is birational if and only if the maps $\gamma_L^M$ and $\gamma_M^G$ are birational.
\end{proposition}

\begin{proof}
There exist $lu_1 \in \orb^M$ with $l \in \orb^L$ and $u_1 \in U \cap M$.
Similarly, $\orb^G$ has representatives of the form $l u_1 u_2$, where $u_2 \in V$ and $l u_1 \in \orb^M \cap  \orb^L (U \cap M)$.

Suppose that $\gamma_L^G$ is birational, then for $l u_1 u_2 \in \orb^G$ as above, we have $C_G(l u_1 u_2) \leq P \leq Q$, so $\gamma_M^G$ is birational by Lemma \ref{lem_bir_crit}.
We show $C_M(l u_1) \leq P$. Let $m \in C_M(l u_1)$, then $mlu_1u_2m^{-1} =  l u_1  m u_2 m^{-1} \in \overline{\orb}^L U \cap \orb^G = \orb^P_{l u_1 u_2}$,  by Lemma \ref{lem_jantzen}. Hence there exists $p \in P$ such that $pm \in C_G(l u_1 u_2) \leq P$. This implies $m \in P$, i.e. $\gamma_L^M$ is birational by Lemma \ref{lem_bir_crit}.

For the other implication, assume $\gamma_L^M$ and $\gamma_M^G$ are birational.
Let $l u_1 u_2 \in \orb^G$ be as above and let $g \in C_G(l u_1 u_2)$.
We show that $g \in P$.
Since $\gamma_M^G$ is birational, then $g \in C_G(l u_1 u_2) \leq Q$ by Lemma \ref{lem_bir_crit}.
So $g=mv$ with $m \in M$ and $v \in V$.
Then $l u_1 u_2 = (mv) (l u_1 u_2) (mv)^{-1}= (m  l u_1 m^{-1})( m \tilde{u}_2 m^{-1})$, where $\tilde{u}_2 \coloneqq (lu_1)^{-1} v (lu_1)u_2 v^{-1} \in V$.
Now $m  l u_1 m^{-1} \in M$ and $ m \tilde{u}_2 m^{-1} \in V$, since $M$ stabilizes $M$ and $V$, and the semi-direct product decomposition of $Q$ yields $m  l u_1 m^{-1} = lu_1$, i.e. $m \in C_M(lu_1)$.
Since $\gamma_L^M$ is also birational, we have $C_M(l u_1) \leq P$, by Lemma \ref{lem_bir_crit}.
Therefore, $g = mv \in P$, i.e. $C_G(l u_1 u_2) \leq P$ and $\gamma_L^G$ is birational, by Lemma \ref{lem_bir_crit}.
\end{proof}

The analogous result for birational induction of nilpotent orbits in Lie algebras is obtained in {\cite[\S 1.2]{Nami2009}} with different techniques.

\subsubsection{Birational rigidity}

\begin{definition} \label{def_uni_bir_rig}
A unipotent conjugacy class in $G$ is said to be \emph{birationally rigid} if it cannot be induced in a birational way from a unipotent class $\orb^L$ inside a proper Levi $L \lneq G$.
\end{definition}

\begin{remark} \label{rk_bir_rig}
By Remark \ref{rk_ind_isogeny} (iii), the orbit $\orb^{\mathfrak{g}} \in \mathcal{N}/G$ is birationally rigid if and only if $\phi(\orb^{\mathfrak{g}}) \in \mathcal{U}/G$ is birationally rigid, where $\phi$ is Springer's isomorphism.
For $\mathfrak{g}$ simple, the complete list of birationally rigid nilpotent orbits can be deduced from \cite{Nami2009, Fu2010}.
By Remark \ref{rk_ind_isogeny} (iv), the notion of birational rigidity does not depend on the isogeny class of the group. 
For a Levi subgroup $L \leq G$ and $\vartheta \in {\rm Aut}(G)$, a class $\orb^L_u \in \mathcal{U}_L/L$ is birationally rigid in $L$ if and only if $\orb^{\vartheta ( L )}_{\vartheta ( u )}$ is birationally rigid in $\vartheta ( L )$.
\end{remark}

\begin{example}
Every rigid conjugacy class is  birationally rigid. The converse is not true: 
if $G$ is simple of type $\mathsf{C}_3$, let $\orb^G$ be the unipotent class $\orb^G$ relative to the partition $[2,2,1,1]$. This is birationally rigid, by \cite[Remark 1.5.2]{Nami2009}. Nonetheless, $\codim_G \orb^G = \dim L_\Theta = \codim_{L_\Theta} \{1\}$, where $\Theta = \{\alpha_2, \alpha_3\}$. Since $\orb^G \in \mathcal{U}/G$ is uniquely determined by its dimension, we have $\orb^G = \ind_{L_\Theta}^G \{1\}$.
\end{example}

\begin{example} \label{eg_type_a}
If $G$ is simple of type $\mathsf{A}$ and $\orb^G \in \mathcal{U}/G$, then $\orb^G$ is rigid if and only if it is birationally rigid if and only if $\orb^G = \{1\}$. By \cite[Theorem 7.2.3]{CollMcGov} and adopting notation therein, if $\orb^G \in \mathcal{U}/G$ corresponds to a partition $\mathbf{p}$ of $n$, then $\orb^G = \ind_L^G \{1\}$, where $L$ corresponds to the  dual partition $\mathbf{p^t}$. The induction is birational, by Lemma \ref{lem_suffbir}. 
Indeed, if $\overline{G}$ denotes the adjoint group in the isogeny class of $G$, we have $C_{\overline{G}}(\nu)$ connected for all $\nu \in \mathcal{N}$, see \cite[Corollary 6.1.6]{CollMcGov}.
\end{example}

\section{Uniqueness of birational induction} \label{s_unique}
In this section we establish an explicit bijection between conjugacy classes in $G$ and a set of data which are ``minimal'' with respect to induction. This will be central in the proof of Theorem \ref{thm_partition}, one of our main results.

\begin{definition}
A pair of \emph{unipotent birational induction data} is $(L, \orb^L)$ where $L \leq G$ is a Levi subgroup, $\orb^L \in \mathcal{U}_L/L$ is birationally rigid and $\mathrm{Ind}_L^G \orb^L$ is birationally induced from $(L, \orb^L)$. We denote by $\mathscr{B}(G)_u$ the set of all unipotent birational data of $G$.
\end{definition}

Notice that $G$ acts on $\mathscr{B}(G)_u$  by simultaneous conjugacy on the pairs and that $\mathscr{B}(G)_u/G$ is finite.
We are going to adapt \cite[Corollary 4.6]{Lo2016} to the case of the conjugacy action of a group on itself.

\begin{lemma} \label{lem_Losev} Let $G$ be reductive and let $\orb^G \in \mathcal{U}/G$. Then there exists, up to $G$-conjugacy, a unique pair $(L, \orb^L) \in \mathscr{B}(G)_u$ such that $\orb^G = \ind_L^G \orb^L$.

In other words, the map
\begin{align*}
\mathscr{B}(G)_u/G &\to \mathcal{U}/G \\
[(L, \orb^L)]_{\sim} &\mapsto \mathrm{Ind}_L^G \orb^L
\end{align*}
is bijective, and all inductions are birational.
\end{lemma}

\begin{proof}
For $G$ reductive, $\mathcal{U} = \mathcal{U}_{[G,G]}$, hence we may assume $G$ semisimple.
Let $G = G_1 \dots G_d$ be the decomposition into simple factors, $d \in \N$.
The decomposition of $G$ carries over to Levi subgroups, parabolic subgroups and unipotent conjugacy classes. For $L \leq G$ a Levi subgroup and $\orb^L \in \mathcal{U}_L/L$, we write $L = \prod_{i =1}^d L_i$ with $L_i$ Levi subgroup of $G_i$ and $\orb^L = \prod_{i =1}^d \orb^{L_i}$ with $\orb^{L_i} \in \mathcal{U}_{L_i}/L_i$.
Remark that Lusztig-Spaltenstein induction in $G$ is compatible with this decomposition.
If $u \in \mathcal{U}$, then $u = u_1 \dots u_d$ with $u_i$ uniquely determined and unipotent in $G_i$ for all $i = 1, \dots, d$ and $C_{G}(u) = \prod_{i=1}^d C_{G_i} (u_i)$. 
The induction $\ind_L^{G} \orb^L$ is birational if and only if all inductions $\ind_{L_i}^{G_i} \orb^{L_i}$ are so.
Also, $\orb^L$ is birationally rigid in $L$ if and only if each $\orb^{L_i}$ is birationally rigid in $L_i$.
Thus, we are reduced to proving the statement for each simple factor $G_i$. This follows from \cite[Corollary 4.6 (i)]{Lo2016} and Springer's isomorphism (Remarks \ref{rk_ind_isogeny}, \ref{rk_bir_rig}).
\end{proof}

In \cite{Lo2016}, Losev described an explicit bijective correspondence between $\mathfrak{g}/G$ and $G$-equivalence classes of birationally minimal induction data, i.e. triples $( \liel,  \zeta, \orb^\mathfrak{l})$ where $\liel$ is a Levi subalgebra of $\lieg$, the orbit $\orb^\mathfrak{l} \in \mathcal{N}_\mathfrak{l}/L$ is birationally rigid and $\zeta \in \mathfrak{z}(\liel)$ is such that the induction $\ind_\liel^\lieg (\zeta + \orb^\mathfrak{l})$ is birational. Our aim is to find an analogue result in the case of the conjugacy action of $G$ on itself. 

\begin{definition} \label{defn_bid}
A triple of \emph{ weakly  birational induction data} is $(M, s, \orb^M)$ where:
$M \leq G$ is a pseudo-Levi subgroup, $s \in Z(M)$ is such that $Z(M)^\circ s$ satisfies \eqref{RP}, $\orb^M \in \mathcal{U}_M/M$ is birationally rigid and $\ind_M^{C_G(s)^\circ} \orb^M$ is birationally induced from $(M, \orb^M)$.
We denote by $\mathscr{B}(G)$  the set of weakly birational induction data of $G$.
\end{definition}

$G$ acts on $\mathscr{B}(G)$ by simultaneous conjugacy on the triples.

\begin{remark} \label{rk_envelope}
Let $(M, s, \orb^M) \in \mathscr{B}(G)$.
Let $L(M) \coloneqq C_G(Z(M)^\circ)$ be the \emph{Levi envelope} of $M$, i.e. the smallest Levi subgroup of $G$ containing $M$ (see \cite[Definition 3.7]{CE1}).
Remark that $C_{L(M)}(s)^\circ = C_G(s)^\circ \cap L(M) = M$, see the proof of Lemma \ref{lem_itsalevi}.
Then we have $\ind_{L(M)}^G (L(M) \cdot (s \orb^M)) = G \cdot (s \ind_M^{C_G(s)^\circ} \orb^M) \eqqcolon \orb^G$ and 
$\orb^G$ is weakly birationally induced  from $(L(M), L(M) \cdot (s \orb^M))$ in the sense of Definition \ref{defn_birind} (b).
For this reason, we will say that  $G \cdot (s \ind_M^{C_G(s)^\circ} \orb^M)$ is weakly birationally induced from $(M, s, \orb^M)$.

When $G$ is semisimple simply connected, we will omit the adverb ``weakly'', i.e. we will say that $\mathscr{B}(G)$ is the set of \emph{birational induction data} of $G$ and that $G \cdot (s \ind_M^{C_G(s)^\circ} \orb^M)$ is birationally induced from $(M, s, \orb^M) \in \mathscr{B}(G)$: this choice agrees with Remark \ref{rk_coincide}.
\end{remark}

Now we prove that every conjugacy class is weakly birationally induced in a unique way from a triple of birational induction data, up to conjugacy.

\begin{theorem} \label{thm_unibirind}
The following map is bijective:
\begin{align*}
\mathscr{B}(G)/G &\to G/G\\
[(M,s, \orb^M)]_{\sim} &\mapsto G \cdot ( s \ind_M^{C_G(s)^\circ} \orb^M ).
\end{align*}

In particular, for $G$ semisimple simply connected, every conjugacy class is birationally induced in a unique way from a triple of birational induction data. 
\end{theorem}

\begin{proof}
We prove surjectivity.
Let $\orb^G = \orb^G_{su}$ with $s \in T$, $u \in K \coloneqq C_G(s)^\circ$.
By Lemma \ref{lem_Losev}, there exists, up to conjugacy in $K$, a unique $( L, \orb^L) \in \mathscr{B}(K)_u$ such that $\orb^{K}_u = \ind_L^{K} ( \orb^L)$. We can assume $T \leq L$ so that   $L$ is a Levi subgroup of $K$ with $Z(K) \leq Z(L)$. Hence, $s \in Z(L)$ and  $Z(L)^\circ s$ satisfies \eqref{RP} for $L$, by Lemma \ref{lem_itsalevi}. In particular, $L = C_G(Z(L)^\circ s)^\circ$ is a pseudo-Levi of $G$ and $(L, s, \orb^L) \in \mathscr{B}(G)$ satisfies $G \cdot (s \ind_L^{C_G(s)^\circ} \orb^L) = \orb^G$.

Now we prove injectivity.
Let $(M_1, s_1, \orb^{M_1}_{u_1}), (M_2, s_2, \orb^{M_2}_{u_2}) \in \mathscr{B}(G)$ and suppose that 
\begin{equation} \label{ind_b}
G \cdot ( s_1 \ind_{M_1}^{C_G(s_1)^\circ} \orb^{M_1}_{u_1} )= G \cdot ( s_2 \ind_{M_2}^{C_G(s_2)^\circ} \orb^{M_2}_{u_2} )
\end{equation}
where the unipotent classes $\ind_{M_i}^{C_G(s_i)^\circ} \orb^{M_i}_{u_i}$ are birationally induced from $(M_i, \orb^{M_i}_{u_i})$ for $i = 1, 2$.
We can assume that $s_1 = s_2 \eqqcolon s \in T$ and set $K \coloneqq C_G(s)^\circ$.
We have that \eqref{ind_b} is equivalent to
\begin{equation*}
C_G(s) \cdot (\mathrm{Ind}_{M_1}^{K} \orb^{M_1}_{u_1}) = C_G(s) \cdot (\mathrm{Ind}_{M_2}^{K} \orb^{M_2}_{u_2}).
\end{equation*}
If $v \in \ind_{M_1}^M \orb^{M_1}_{u_1}$, then there exists $g \in C_G(s) $ such that $gvg^{-1} \in \ind_{M_2}^{K} \orb^{M_2}_{u_2}$.
Write $g = w^{-1}h$ for suitable $h \in M$ and $w \in N_G (T) \cap C_G(s)$, and up to choosing $hvh^{-1}$ instead of $v$ as a representative, we can assume that $g = w^{-1} \in N_G \left(T \right) \cap C_G(s) $.
Therefore, we have:
$$v \in w (\ind_{M_2}^{K} \orb^{M_2}_{u_2}) w^{-1} = \ind_{w M_2 w^{-1}}^{w K w^{-1}} \orb^{w M_2 w^{-1}}_{w{u_2} w^{-1}}.$$
Since $w$ acts as an automorphism of $K$, the induction is birational (Remark \ref{rk_conj_data_bir}) and $\orb^{wM_2 w^{-1}}_{w  u_2 w^{-1}}$ is birationally rigid (Remark \ref{rk_bir_rig}), it follows that
$$ v \in \ind_{M_1}^{K} \orb^{M_1}_{u_1} \cap \ind_{w M_2 w^{-1}}^{K} \orb^{w M_2 w^{-1}}_{w u_2 w^{-1}}.$$
By Lemma \ref{lem_Losev}, the pairs $ ( M_1, \orb^{M_1}_{u_1})$ and
$ ( w M_2 w^{-1}, \orb^{w M_2 w^{-1}}_{w u_2 w^{-1}})$ 
are conjugate in $K$, consequently  $ (M_1, s_1, \orb^{M_1}_{u_1}) $ and 
  $ (M_2, s_2, \orb^{M_2}_{u_2})$ are conjugate in $G$ via $g'w^{-1}$ for some $g' \in K$.
  
The last statement is a consequence of the proof together with Remark \ref{rk_coincide}. 
\end{proof}

\section{Jordan classes and birational sheets} \label{s_jcbs}
We recall the notions of Jordan classes in a reductive group, introduced in \cite{LusztigICC} and we collect some results on sheets from \cite[§4]{CE1}. After that, we define birational closures of Jordan classes and birational sheets of a semisimple simply connected group $G$.

\begin{definition}
The set of \emph{decomposition data} of $G$ is the set of all triples $\tau = (M, Z(M)^\circ z, \orb^M)$ where $M \leq G$ is a pseudo-Levi subgroup, $Z(M)^\circ z$ is a connected component of $Z(M)$ satisfying \eqref{RP} for $M$ and $\orb^M \in \mathcal{U}_M/M$. We denote by $\mathscr{D}(G)$ the set of decomposition data of $G$.
\end{definition}

The group $G$ acts on $\mathscr{D}(G)$ by simultaneous conjugacy on the triples.
We associate to any $su \in G$ its decomposition data $( C_G(s)^\circ, Z(C_G(s)^\circ)^\circ s, \orb^{C_G(s)^\circ}_u )$.

\begin{definition}
Two elements $su, rv \in G$ are \emph{Jordan equivalent} if their decomposition data are conjugate in $G$.
We denote with $J (su)$  the \emph{Jordan class}  of $su$, i.e. the equivalence class of all elements of $G$ which are Jordan equivalent to $su$.
\end{definition}

We have
$J ( su ) = G \cdot (( Z( C_G(s)^\circ )^\circ s )^{reg} \orb^{C_G(s)^\circ}_u )$. We denote by $\mathscr{J} (G)$ the set of Jordan classes in $G$.
The group $G$ is partitioned into finitely many Jordan classes, which are in one-to-one correspondence with $\mathscr{D}(G)/G$.
If $\tau = (M, Z(M)^\circ z, \orb^M) \in \mathscr{D}(G)$, then $J(\tau) \coloneqq G \cdot ((Z(M)^\circ z)^{reg} \orb^M)$.
Jordan classes are smooth irreducible locally closed subvarieties of $G$, they are unions  of equidimensional conjugacy classes, this is stated in \cite[\S 3.1]{LusztigICC}, see also \cite[Corollary 3.8.1]{Broer_Lectures} and \cite[Proposition 2.3]{broer_1998}.
The closure of a Jordan class is a union of Jordan classes and the same holds for its regular part.
If $\tau = (M, Z(M)^\circ z, \orb^M) \in \mathscr{D}(G)$, we have:
\begin{align*}
\overline{J(\tau)} = \bigcup_{t \in Z(M)^\circ z} \overline{G \cdot (t \ind_M^{C_G(t)^\circ} \orb^M)}, \qquad 
\overline{J(\tau)}^{reg} = \bigcup_{t \in Z(M)^\circ z} G \cdot (t \ind_M^{C_G(t)^\circ} \orb^M).
\end{align*}
Since $\overline{J(\tau)}^{reg}$ is open in $\overline{J(\tau)}$, it is irreducible and locally closed in $G$.

The sheets for the action of $G$ on itself under conjugation are unions of Jordan classes. For each sheet $S$ there exists a unique Jordan class $J \in \mathscr{J}(G)$ such that $J$ is dense in $S$ and $S = \overline{J}^{reg}$. A Jordan class $J$ is dense in a sheet if and only if $J = J(\tau)$ where $\tau = (M, Z(M)^\circ z, \orb^M) \in \mathscr{D}(G)$ where $\orb^M \in \mathcal{U}_M/M$ is rigid, see \cite[Theorem 5.6 (a)]{CE1}.

\begin{remark} \label{rk_deco_class}
Jordan (or decomposition) classes and sheets were also defined for the adjoint action of $G$ on $\lieg$. Decomposition classes in $\lieg$ are parametrized by $G$-conjugacy classes of pairs $(\mathfrak{l}, \orb^\mathfrak{l})$ where $\mathfrak{l}$ is a Levi subalgebra of $\mathfrak{g}$ and $\orb^\mathfrak{l} \in \mathcal{N}_\mathfrak{l}/L$. If $(\mathfrak{l}, \orb^\mathfrak{l})$ is such a pair, it parametrizes the Jordan class \begin{equation} \label{eq_deco_class}
\mathfrak{J}(\mathfrak{l}, \orb^\mathfrak{l}) \coloneqq {\rm Ad}(G)(\mathfrak{z(l)}^{reg} + \orb^\mathfrak{l}).
\end{equation}
Sheets are regular closures of Jordan classes $\overline{\mathfrak{J}(\mathfrak{l}, \orb^\mathfrak{l})}^{reg}$ with $\orb^\mathfrak{l}$ a rigid nilpotent orbit in $\mathfrak{l}$.
Definitions and results can be found in \cite{BK79, Borho, broer_1998, Broer_Lectures}.
\end{remark}

There exist sheets in simple Lie algebras which intersect non-trivially, see \cite[§6.6]{BK79} and \cite[§7.4]{Borho}. Sheets are disjoint in simple Lie algebras of type $\mathsf{A}$, \cite{Dix}.
For $\mathfrak{g}$ simple of classical type all sheets are smooth, see \cite{imhof}; if $\mathfrak{g}$ is simple exceptional there exist singular sheets, see \cite{bulois} for the list of smooth ones.
Similarly, in the case of a simple group $G$, there exist non-smooth sheets and distinct sheets with non-empty intersection. If $G = {\rm SL}_n(\C)$, all sheets are smooth, see \cite[§6.3]{ACE}.

\subsection{Preliminary constructions} \label{ss_noname}
Let $\mathfrak{l}$ be a Levi subalgebra of $\mathfrak{g}$ and let $\orb^{\mathfrak{l}} \in \mathcal{N}_\mathfrak{l}/L$. In \cite[§4]{Lo2016}, for any such pair $(\mathfrak{l}, \orb^{\mathfrak{l}})$ Losev defines the set 
$\bir(\mathfrak{z(l)}, \orb^{\mathfrak{l}}) = \{ \zeta \in \mathfrak{z(l)} \mid \ind_\mathfrak{l}^{\mathfrak{g}} ( \zeta + \orb^{\mathfrak{l}}) \mbox{ is birational}\} = \{ \zeta \in \mathfrak{z(l)} \mid \ind_\mathfrak{l}^{\mathfrak{c_g}(\zeta)} \orb^{\mathfrak{l}} \mbox{ is birational}\}$, by Remark \ref{rk_lie}. 
In particular, $\bir(\mathfrak{z(l)}, \orb^{\mathfrak{l}}) $ only depends on the pair $(\mathfrak{l}, \orb^\mathfrak{l})$ and not on the parabolic subgroup chosen for the generalized Springer map.
We would like to define a similar object for the group case, but Lemma \ref{lem_gsp_conn} and Remarks \ref{rk_coincide} and \ref{rk_envelope} suggest two distinct approaches.

Let $ \tau \coloneqq (M, Z(M)^\circ s, \orb^M) \in \mathscr{D}(G)$ and let $L(M)$ be the Levi envelope of $M$ in $G$. 
To $\tau \in \mathscr{D}(G)$, we associate the two sets:
\begin{align} 
\bir ( Z(M)^\circ s , \orb^M ) &\coloneqq \{ z \in Z(M)^\circ s \mid  \ind_{L(M)}^G (L(M) \cdot (z\orb^M)) \mbox{ is birationally induced}\}; \label{eq_centre} \\
\wbir( Z(M)^\circ s , \orb^M ) &\coloneqq \{ z \in Z(M)^\circ s \mid  \ind_{M}^{C_G(z)^\circ} \orb^M \mbox{ is birationally induced}\}= \label{eq_wcentre}\\
& =\{ z \in Z(M)^\circ s \mid  \ind_{L(M)}^{G} (L(M) \cdot (z\orb^M)) \mbox{ is weakly birationally induced}\}. \notag
\end{align} 

\begin{remark} \label{rk_centre}
The definition in \eqref{eq_wcentre} does not depend on the choice of a parabolic subgroup, thanks to the arguments in \S \ref{sss_indeparab}.
Moreover, if $z \in (Z(M)^\circ s)^{reg}$, then $C_G(z)^\circ = M$, hence $\varnothing \neq (Z^\circ s)^{reg} \subset \wbir ( Z(M)^\circ s , \orb^M )$.
By Remarks \ref{rk_coincide} and \ref{rk_envelope}, we have $\bir ( Z(M)^\circ s , \orb^M ) \subset \wbir( Z(M)^\circ s , \orb^M )$.
As a consequence of Lemma \ref{lem_gsp_conn}, when $G$ is semisimple simply connected, $\bir ( Z(M)^\circ s , \orb^M ) = \wbir( Z(M)^\circ s , \orb^M )$; on the other hand, when $G$ is not simply connected, the inclusion can be proper as the following example shows.
\end{remark}

\begin{example} \label{eg_psl2}
Consider $G = \mathrm{SL}_2(\C)$, let $\overline{G} = \mathrm{PSL}_2(\C)$ and let $\pi \colon G \to \overline{G}$, $\pi(g) = \bar{g}$ be the central isogeny.
Let $T \leq G$ be the torus consisting of diagonal matrices, let $B \leq G$ be the Borel subgroup individuated by the upper triangular ones, let $U$ be the unipotent radical of $B$ and let $\overline{B} = \pi(B)$ and $\overline{U} = \pi(U)$.
Let $\tau = (\overline{T}, \overline{T}, \{ \bar{e}\}) \in \mathscr{D}(\overline{G})$ and let $s \coloneqq {\rm diag}[i, -i]$.
Then one can verify that $\bar{s} \in \wbir(\overline{T}, \{ \bar{e}\}) \setminus \bir(\overline{T}, \{ \bar{e}\})$. 
The
 generalized Springer map \eqref{gsp} reads:
 \begin{equation*}
 \gamma \colon \overline{G} \times^{\overline{B}} (\bar{s} \overline{U}) \to \overline{G} \cdot (\bar{s} \overline{U}) = \orb^{\overline{G}}_{\bar{s}}.
 \end{equation*}
This map is not birational, because $C_{\overline{G}} (\bar{s}) =  N_{\overline{G}}(\overline{T}) \lneq \overline{B}$, namely $\gamma$ is a $2:1$ covering; on the other hand, $\gamma_{\bar{s}}$ as in \eqref{ugsp} is trivially birational.
\end{example}

We now describe the structure of the set $\wbir(Z(M)^\circ s, \orb^M)$ for a pseudo-Levi $M \leq G$ and $\orb^M \in \mathcal{U}_M/M$.
Thanks to Remark \ref{rk_centre}, when $G$ is semisimple simply connected, the results will give a description of the set  $\bir(Z(M)^\circ s, \orb^M)$.

Recall that an element $su \in G$ is \emph{isolated} if $C_G(Z(C_G(s)^\circ)^\circ) = G$, see \cite[Definition 2.6]{LusztigICC}.

\begin{remark}\label{rk_cosettori_centr}
Let $M \leq G$ be a pseudo-Levi subgroup, set $Z \coloneqq Z(M)$ and let $s \in Z$ such that $Z^\circ s$ satisfies \eqref{RP} for $M$.
Define the set $$\mathcal{Z}(Z(M)^\circ s) \coloneqq \{ Z(C_G(z)^\circ)^\circ z \mid z \in Z^\circ s \}.$$
We prove that this set is finite.
We can assume $T \leq M$, so that $z \in Z(M) \leq T$ hence $C_G(z)^\circ = \langle T, U_{\pm \alpha} \mid \alpha \in \Phi_z \rangle$ with finitely many possibilities for the root subsystem $\Phi_z = \{\alpha \in \Phi \mid \alpha(z) = 1 \rangle \subset \Phi$, moreover the connected components of $Z(C_G(z)^\circ)$ are finitely many.

We write $\mathcal{Z}(Z(M)^\circ s) = \{ Z(M_i)^\circ s_i \mid i \in I\}$ for a finite index set $I$ and suitable $s_i \in Z^\circ s$ with $C_G(s_i)^\circ = M_i$.
Thanks to Lemma \ref{lem_itsalevi}, we can define a map:
\begin{align*}
\mathcal{Z}(Z(M)^\circ s) &\to \{ M_i \leq G \mid M \mbox{ is a Levi subgroup of } M_i \} \\
Z(M_i)^\circ s_i  & \mapsto C_G(Z(M_i)^\circ s_i)^\circ = M_i.
\end{align*}
The set $\mathcal{Z}(Z(M)^\circ s) $ is partially ordered by inclusion: if $T \leq M_1 \leq M_2$, with $M_1, M_2$ pseudo-Levi subgroups of $G$,
then $Z(M_2) \leq Z(M_1)$ and $Z(M_2)^\circ z \subset Z(M_1)^\circ z$ for all $z \in Z(M_2)$.
The maximum of $\mathcal{Z}(Z(M)^\circ s)$ is $Z(M)^\circ s$ and its minimal elements are connected components $Z(M_i)^\circ s_i$ consisting of isolated elements.
The above map reverses inclusions: $Z(M_i)^\circ s_i \subset Z(M_j)^\circ s_j$ implies $M_j \leq M_i$.
\end{remark}

\begin{lemma} \label{lem_bir_strata}
Let $(M, Z(M)^\circ s, \orb^M) \in \mathscr{D}(G)$, set $Z \coloneqq Z(M)$.
Let $z \in Z^\circ s$.
\begin{enumerate}[noitemsep, nolistsep]
\item[(i)] If $\ind_M^{C_G(z)^\circ} \orb^M$ is birationally induced from $(M, \orb^M$), then $\ind_M^{C_G(z')^\circ} \orb^M$ is birationally induced from $(M, \orb^M$) for all $z' \in Z^\circ s$ such that $C_G(z')^\circ \leq C_G(z)^\circ$.
\item[(ii)] If $\ind_M^{C_G(z)^\circ} \orb^M$ is not birationally induced from $(M, \orb^M)$, then $\ind_M^{C_G(z')^\circ} \orb^M$ is not birationally induced from $(M, \orb^M)$ for all $z' \in Z^\circ s$ such that $C_G(z')^\circ \geq C_G(z)^\circ$.
\end{enumerate}
\end{lemma}

\begin{proof}
(i) follows from Proposition \ref{prop_trans} and Remark \ref{rk_cosettori_centr}.
We prove (ii) by contradiction. We can assume $M = C_G(s)^\circ$ for $s\in T$.
Suppose that there is $z' \in Z(C_G(z)^\circ)^\circ z \cap \wbir(Z^\circ s, \orb^M)$.
Then $C_G(z')^\circ \geq C_G(z)^\circ$ and $\ind_M^{C_G(z')^\circ} \orb^M$ is birationally induced from $(M, \orb^M)$.
This contradicts Proposition \ref{prop_trans}.
\end{proof}

\begin{corollary}  Let $(M, Z(M)^\circ s, \orb^M) \in \mathscr{D}(G)$, set $Z \coloneqq Z(M)$.
If $z \in \wbir(Z^\circ s, \orb^M)$ for all isolated $z \in Z^\circ s$, then $\wbir(Z^\circ s, \orb^M) = Z^\circ s$.
\end{corollary}

\begin{proof}
Let $z \in Z^\circ s$ be isolated, with $\ind_M^{C_G(z)^\circ} \orb^M$ birationally induced from $(M, \orb^M)$. Let $Z(C_G(z)^\circ)^\circ z = Z(M_i)^\circ s_i \in \mathcal{Z}(Z^\circ s, \orb^M)$. Then all $z' \in Z^\circ s$ such that $Z(C_G(z')^\circ)^\circ z'$ is greater or equal to $Z(M_i)^\circ s_i$ in the partial order on $\mathcal{Z}(Z^\circ s)$ satisfy $\ind_M^{C_G(z')^\circ} \orb^M$ is birationally induced from $(M, \orb^M)$. This is a consequence of Lemma \ref{lem_bir_strata} and Remark \ref{rk_cosettori_centr}. Since isolated elements are minimal elements of $\mathcal{Z}(Z^\circ s)$, the statement follows.
\end{proof}

\begin{remark} \label{rk_not_bir}
For $(M, Z(M)^\circ s, \orb^M) \in \mathscr{D}(G)$, set $Z \coloneqq Z(M)$ and define $$\mathcal{Z}'(Z^\circ s, \orb^M) \coloneqq \{Z(M_i)^\circ s_i \in \mathcal{Z}(Z^\circ s) \mid \ind_M^{M_i} \orb^M \mbox{ is not birational} \}.$$
By Remark \ref{rk_cosettori_centr}, this set is finite and by Lemma \ref{lem_bir_strata}, it is a subposet of $\mathcal{Z}(Z^\circ s)$.
Let $\hat{\mathcal{Z}}(Z^\circ s, \orb^M)$ be the subset of maximal elements in $\mathcal{Z}'(Z^\circ s, \orb^M)$.

Now we prove that the set $\wbir(Z^\circ s, \orb^M)$ is the complement in $Z^\circ s$ of the finite union (possibly empty) of shifted tori which are elements of $\mathcal{Z}'(Z^\circ s, \orb^M)$.
\end{remark}

\begin{proposition} \label{prop_birzl_gp}
Retain the notation from Remark \ref{rk_not_bir}.
Then\footnote{Notice that $\bigcup_{\hat{\mathcal{Z}}(Z^\circ s, \orb^M)} Z(M_i)^\circ s_i =\bigcup_{\mathcal{Z}'(Z^\circ s, \orb^M)} Z(M_i)^\circ s_i$, but the first union has the advantage to be expressed with a minimal number of elements.}
$$Z^\circ s \setminus \wbir(Z^\circ s, \orb^M) = 
\bigcup_{\hat{\mathcal{Z}}(Z^\circ s, \orb^M)} Z(M_i)^\circ s_i.$$

In particular, $\wbir(Z^\circ s, \orb^M)$ is open in $Z^\circ s$.
\end{proposition}

\begin{proof}
One inclusion is clear, as $\hat{\mathcal{Z}}(Z^\circ s, \orb^M) \subset \mathcal{Z}'(Z^\circ s, \orb^M)$, hence $z \in Z(M_i)^\circ s_i$ implies  that $\ind_M^{C_G(z)^\circ} \orb^M$ is not birationally induced from $(M, \orb^M)$, i.e. $z \in Z^\circ s \setminus \wbir(Z^\circ s, \orb^M)$.

For the other inclusion, let $z \in Z^\circ s \setminus \wbir(Z^\circ s, \orb^M)$.
Then $z$ satisfies the assumptions of Lemma \ref{lem_bir_strata} (b), hence $Z(C_G(z)^\circ)^\circ z \in \mathcal{Z}'(Z^\circ s, \orb^M)$.
By Remark \ref{rk_not_bir} there exists a maximal $Z(M_i)^\circ s_i \in \hat{\mathcal{Z}}(Z^\circ s, \orb^M)$ containing $Z(C_G(z)^\circ)^\circ z$, and this ends the proof.
\end{proof}

\subsection{Birational closures of Jordan classes} \label{ss_bcjc}
In this part, we assume $G$ semisimple simply connected and we apply the results obtained above to define and birational closures of Jordan classes and study their structure.

\begin{definition} \label{defn_birclos}
Let $su \in G$, the \emph{birational closure} of $J(su)$ is:
$$\overline{J(su)}^{bir} \coloneqq \bigcup_{z \in \bir(Z(C_G(s))^\circ s, \orb^{C_G(s)}_u)} G \cdot \left( z \ind_{C_G(s)}^{C_G(z)} \orb^{C_G(s)}_u \right),$$ where $\bir(Z(C_G(s))^\circ s, \orb^{C_G(s)}_u)$ is as in  \eqref{eq_centre}, equivalently as in \eqref{eq_wcentre}.
\end{definition}

\begin{remark}
For $su \in G$, the birational closure $\overline{J(su)}^{bir}$ is $G$-stable by construction.
We have  $J(su) \subset \overline{J(su)}^{bir} \subset \overline{J(su)}^{reg}$, always by construction.
\end{remark}

\begin{lemma} \label{lem_wdbirclos}
Definition \ref{defn_birclos} is independent of the representative of the Jordan class. In particular, for $\tau \in \mathscr{D}(G)$, the set $\overline{J(\tau)}^{bir}$ is well-defined and if $\tau$ and $\tau'$ represent the same class in $\mathscr{D}(G)/G$, then $\overline{J(\tau)}^{bir} = \overline{J(\tau')}^{bir}$.
\end{lemma}

\begin{proof}
We show that $J(su) = J(rv)$ implies $\overline{J(su)}^{bir} = \overline{J(rv)}^{bir}$.
Let $s_1 u_1 \in \overline{J(su)}^{bir}$,
namely we can assume that $s_1 \in Z(C_G(s))^\circ s$
and that $\orb^{C_G(s_1)}_{u_1} = \ind_{C_G(s)}^{C_G(s_1)} \orb^{C_G(s)}_u$ is birationally induced from $(C_G(s), \orb^{C_G(s)}_u)$.
By hypothesis, $(C_G(s), Z(C_G(s))^\circ s, \orb^{C_G(s)}_u)$
 and $(C_G(r), Z(C_G(r))^\circ r, \orb^{C_G(r)}_v)$ are conjugate by an element $g \in G$.
 Hence, $g s_1 g^{-1} \in Z(C_G(r))^\circ r$ and
\begin{align*}
g u_1 g^{-1} \in 
g (\ind_{C_G(s)}^{C_G(s_1)} \orb^{C_G(s)}_u) g^{-1}
&= \ind_{g C_G(s) g^{-1}}^{g C_G(s_1) g^{-1}} (g \orb^{C_G(s)}_u g^{-1}) =\\
&=\ind_{C_G(r)}^{C_G(g s_1 g^{-1})} \orb^{C_G(r)}_{v},
\end{align*}
which is birationally induced from $(C_G(r), \orb^{C_G(r)}_{v})$ by Remark \ref{rk_conj_data_bir}.
This yields $g s_1 u_1 g^{-1} = (g s_1 g^{-1})(g u_1 g^{-1}) \in \overline{J(rv)}^{bir}$, the proof follows from $G$-stability of $\overline{J(rv)}^{bir}$.
\end{proof}

We continue with other structural results on birational closures of Jordan classes.


\begin{corollary}
Let $J_1, J_2 \in \mathscr{J}(G)$. If $J_1 \subset \overline{J_2}^{bir}$, then $\overline{J_1}^{bir} \subset \overline{J_2}^{bir}$.
\end{corollary}

\begin{proof}
This follows from Definition \ref{defn_birclos} and Proposition \ref{prop_trans}.
\end{proof}

\begin{proposition} \label{prop_birjord}
Let $J \in \mathscr{J}(G)$. Then $\overline{J}^{bir}$ is obtained from $\overline{J}^{reg}$ by neglecting a finite number of regular closures of other Jordan classes of $G$, and it is open in $\overline{J}^{reg}$.
In particular, birational closures of Jordan classes are irreducible locally closed subsets of $G$.
\end{proposition}

\begin{proof}
Let $J \coloneqq J(\tau)$ with $\tau = (M, Z(M)^\circ s, \orb^M) \in \mathscr{D}(G)$ and let $Z \coloneqq Z(M)$.
Consider the set $\hat{\mathcal{Z}}(Z^\circ s, \orb^M)$ defined in Remark \ref{rk_not_bir},  and write $\hat{\mathcal{Z}}(Z^\circ s, \orb^M) = \{ Z(M_i)^\circ s_i \mid i \in \hat{I}\}$, for a suitable index set $\hat{I}$. 
As in the proof of Proposition \ref{prop_birzl_gp}, we have $\bir(Z^\circ s, \orb^M) = Z^\circ s \setminus \bigcup_{i \in \hat{I}} Z(M_i)^\circ s_i$.
Since $Z^\circ s = Z^\circ s_i$  for all $i \in \hat{I}$, Lemma \ref{lem_itsalevi} implies that $Z^\circ s_i$ satisfies \eqref{RP} for $M$ and $C_G(s)$ is a Levi subgroup of $C_G(s_i)$.
Then:
$$\overline{J}^{reg} \setminus \overline{J}^{bir} = \bigcup_{i \in \hat{I}} \left( \bigcup_{t \in Z(C_G(s_i))^\circ s_i} G \cdot ( t \ind_{M}^{C_G(t)} \orb^M ) \right).$$
Set
$\orb^{C_G(s_i)} \coloneqq \ind_{M}^{C_G(s_i)} \orb^M$ and
 $\tau_i \coloneqq ( C_G(s_i),
 Z(C_G(s_i))^\circ s_i,
  \orb^{C_G(s_i)}) \in \mathscr{D}(G)$ for all $i \in\hat{I}$.
  Then, by transitivity of induction:
   $$\bigcup_{t \in Z(C_G(s_i))^\circ s_{i}} G \cdot ( t \ind_{M}^{C_G(t)} \orb^M )
 = \bigcup_{t \in Z(C_G(s_i))^\circ s_i} G \cdot ( t \ind_{C_G(s_i)}^{C_G(t)}  \orb^{C_G(s_i)} ) = \overline{J(\tau_i)}^{reg}.$$
If $J \subset G_{(n)}$, then $\overline{J}^{reg}, \overline{J}^{bir} \subset G_{(n)}$, hence also  $\overline{J}^{reg} \setminus \overline{J}^{bir} \subset G_{(n)}$, and $\overline{J}^{reg} \setminus \overline{J}^{bir} = \bigcup_{i \in I} \overline{J(\tau_i)}^{reg} = \left( \bigcup_{i \in I} \overline{J(\tau_i)} \right) \cap G_{(n)}$, 
so that $\overline{J}^{reg} \setminus \overline{J}^{bir} $ is closed in $G_{(n)}$.
As a consequence, $\overline{J}^{reg} \setminus \overline{J}^{bir} $ is closed in $\overline{J}^{reg}$. The last assertion follows since $\overline{J}^{reg}$ is an irreducible locally closed subset of $G$.
\end{proof}

\begin{corollary} \label{cor_unijord}
If $J \in \mathscr{J}(G)$, then $\overline{J}^{bir}$ is a union of Jordan classes in $G$.
\end{corollary}
\begin{proof}
For $J \in \mathscr{J}(G)$, we have that $\overline{J}^{reg}$ and $\overline{J}^{reg} \setminus \overline{J}^{bir}$ are unions of Jordan classes of $G$ (see \S 5 and Proposition \ref{prop_birjord}), hence so is $\overline{J}^{bir}$.
\end{proof}

\begin{remark} \label{rk_bircloslie} Consider the adjoint action of $G$ on $\mathfrak{g}$.
Let $\Theta \subset \Delta$ and let $\mathfrak{l} \coloneqq \Lie(L_\Theta)$ be a standard Levi subalgebra of $\mathfrak{g}$.
 Then we can define the sets $\mathcal{L}_\mathfrak{l} \coloneqq \{ \mathfrak{m} \mbox{ Levi subalgebra of } \mathfrak{g} \mid \mathfrak{m} \supset \mathfrak{l} \}$ and $\mathcal{Z}_{\mathfrak{l}} \coloneqq \{ \mathfrak{z(m)} \mid \mathfrak{m} \in \mathcal{L}_\mathfrak{l}\}$: inclusion of subalgebras endows $\mathcal{L}_\mathfrak{l}$ and $\mathcal{Z}_\mathfrak{l}$ with a poset structure.
The set $\mathcal{L}_\mathfrak{l}$ is finite, as $\mathfrak{m} \in \mathcal{L}_\mathfrak{l}$ has the form $\mathfrak{m} = \mathfrak{h} \oplus \bigoplus_{\alpha \in \Xi} \mathfrak{g}_\alpha$, where $\Xi$ is a root subsystem of $\Phi$ such that $\Z \Theta \cap \Phi \subset \Xi$, and there are finitely many possibilities for $\Xi$.
The map $\mathfrak{m} \mapsto \mathfrak{z(m)}$ gives an anti-isomorphism of posets $\mathcal{L}_{\mathfrak{l}}  \to \mathcal{Z}_{\mathfrak{l}}$ whose inverse is $\mathfrak{k} \mapsto \mathfrak{c_g(k)}$, see \cite[Theorem 29.5.7]{ty}.
Readapting proofs of Lemma \ref{lem_bir_strata}, Propositions \ref{prop_birzl_gp}, \ref{prop_birjord} and Corollary \ref{cor_unijord}, one gets the following results, already obtained in \cite[§4]{Lo2016}.
If $\orb^\mathfrak{l} \in \mathcal{N}_\mathfrak{l}/L$, the set $\bir(\mathfrak{z(l)}, \orb^\mathfrak{l})$ is the complement to a finite (possibly empty) union of elements in $\mathcal{Z}_\mathfrak{l}$, it always contains $\mathfrak{z(l)}^{reg}$, it is open in $\mathfrak{z(l)}$ and it coincides with the whole $\mathfrak{z(l)}$ if and only if it contains $0$.
For any pair $(\mathfrak{l}, \orb^\mathfrak{l})$ where $\mathfrak{l} \subset \mathfrak{g}$ is a Levi subalgebra and $\orb^\mathfrak{l} \in \mathcal{N}_\mathfrak{l}/L$ one can define the birational closure of the decomposition class $\mathfrak{J}(\mathfrak{l}, \orb^\mathfrak{l})$ (see \eqref{eq_deco_class} in Remark \ref{rk_deco_class}):
$$\overline{\mathfrak{J}(\mathfrak{l}, \orb^\mathfrak{l})}^{bir} \coloneqq \bigcup_{\zeta \in \bir(\mathfrak{z(l)},\orb^\mathfrak{l})} \ind_{\mathfrak{l}}^{\mathfrak{g}} (\zeta + \orb^\mathfrak{l}).$$
This is open in $\overline{\mathfrak{J}(\mathfrak{l}, \orb^\mathfrak{l})}^{reg}$ and in $\overline{\mathfrak{J}(\mathfrak{l}, \orb^\mathfrak{l})}$, hence it is a locally closed irreducible subset of $\mathfrak{g}$ and it is a birational sheet of $\mathfrak{g}$ if and only if $\orb^\mathfrak{l}$ is birationally rigid in $\mathfrak{l}$.
\end{remark}

\subsection{Birational sheets} \label{ss_bs}
In this section, we still assume $G$  semisimple and simply connected and we prove one of the main result of the work: inspired by \cite[\S 4]{Lo2016}, we define birational sheets for the conjugation action of $G$ on itself and we prove that they partition $G$.

We start by defining the set: 
\begin{equation} \label{eq_bbg}
\mathscr{BB}(G) \coloneqq \{ \tau = (M, Z(M)^\circ s, \orb^M) \in \mathscr{D}(G) \mid \orb^M \mbox{ birationally rigid} \}.
\end{equation}
$G$ acts on $\mathscr{BB}(G)$ by simultaneous conjugacy and  $\mathscr{BB}(G)/G$ is finite because $\mathscr{D}(G)/G$ is so.

\begin{definition} \label{defn_birsh}
For $\tau \in \mathscr{BB}(G)$, the \emph{birational sheet} associated to $\tau$ is: $S(\tau)^{bir} = \overline{J(\tau)}^{bir}$. 
\end{definition}

\begin{lemma} \label{lem_birsh} Let $\tau \in \mathscr{BB}(G)$. Then $S(\tau)^{bir}$ is a $G$-stable irreducible locally closed subvariety of $G$ and decomposes as a union of Jordan classes. Birational sheets of $G$ are in one-to-one correspondence with the finite set $\mathscr{BB}(G)/G$.
\end{lemma}

\begin{proof}
This follows from Definition \ref{defn_birsh}, Proposition \ref{prop_birjord} and Corollary  \ref{cor_unijord}.
\end{proof}

\begin{theorem} \label{thm_partition} 
Let $G$ be semisimple simply connected. Then the birational sheets of $G$ form a partition of $G$.
\end{theorem}

\begin{proof}
We prove that $\orb^G_{su} \in G/G$ belongs to a unique birational sheet.
By Theorem \ref{thm_unibirind}, $\orb^G_{su}$ is birationally induced in a unique way from $(M, s, \orb^M) \in \mathscr{B}(G)$ up to conjugation in $G$.
This triple uniquely determines the class of $\tau = (M, Z(M)^\circ s, \orb^M)$ in $\mathscr{BB}(G)/G$.
Hence, $\orb^G_{su} \subset S(\tau)^{bir}$ and $S(\tau)^{bir}$ is unique.
\end{proof}

Every birational sheet, being irreducible and contained in $G_{(n)}$ for some $n \in \N$, is contained in a sheet.

\begin{example} \label{eg_c2}
In general, a sheet is not an union of birational sheets.
Let $G = {\rm Sp}_4(\C)$, let $\Theta_i = \{\alpha_i\}$ for $i =1,2$ 
and let $L_i = L_{\Theta_i}$. Let $\orb^G_{subreg} \subset \mathcal{U}$ be the subregular class, then $\orb^G_{subreg} = \ind_{L_i}^{G} \{1\}$ for $i=1,2$, but it is birationally induced only from $(L_1, \{1\})$. 
 Let $\tau_i = (L_i, Z(L_i)^\circ, \{1 \})$
  for $i = 1,2$, 
  then $\overline{J(\tau_1)}^{bir} = \overline{J(\tau_1)}^{reg}$ but $\overline{J(\tau_2)}^{reg} = \overline{J(\tau_2)}^{bir} \cup \orb^G_{subreg}$, where
  $\overline{J(\tau_2)}^{bir}$ is a birational sheet, whereas $\orb^G_{subreg}$ is not so.
\end{example}

All sheets of $\mathfrak{g}$ contain nilpotent orbits \cite[§3.2]{Borho}, but not all birational sheet of $\mathfrak{g}$ do \cite[§4]{Lo2016}. Similarly, all sheets of $G$ contain isolated classes, see \cite[Proposition 3.1]{CarnoBul}, but we give an  example of a birational sheet without this property.

\begin{example} \label{eg_c3} 
Let $G = {\rm Sp}_6(\C)$, let $s_a = {\rm diag}(1,a,-1,-1,a^{-1},1) \in T$. Fix $\bar{a} \in \C \setminus \{-1,0,1\}$ and set $\bar{s} \coloneqq s_{\bar{a}}$ and $M \coloneqq C_G(\bar{s})$. If $\tau = (M, Z(M)^\circ \bar{s}, \{1\})$, then:
\begin{align*}
\overline{J(\tau)}^{reg} =  \bigcup_{z \in Z(M)^\circ \bar{s}} G \cdot (z \ind_M^{C_G(z)} \{1\}) =  \left( \bigcup_{a \in \C \setminus \{-1,0,1\}} \orb^G_{s_a} \right) \cup \orb^G_{s_1 u} \cup \orb^G_{s_{-1} v}
\end{align*}
where the first member is $J(\tau)$ while $\orb^G_{s_1 u}$ and $\orb^G_{s_{-1} v}$ are the two isolated classes of the sheet $\overline{J(\tau)}^{reg}$, indeed $C_G(s_1)$ and  $C_G(s_{-1})$ are semisimple of type $\mathsf{C}_1 \mathsf{C}_2$.
Decompose $C_G(s_1)= K' K''$, where $K' \simeq {\rm Sp}_4(\C)$ and $K'' \simeq {\rm Sp}_2(\C)$ and decompose $M = M' K''$, with $M'  \simeq \C^\times \times {\rm Sp}_2(\C) \leq K'$.
Then $\orb^{C_G(s_1)}_u = \ind_M^{C_G(s_1)} \{1\} \simeq \orb^{K'}_{subreg} \times \{1\}$, where $\{1\}$ is the trivial class in $K''$ and $\orb^{K'}_{subreg} = \ind_{M'}^{K'} \{1\}$ is the subregular class in $\mathcal{U}_{K'}$. By Example \ref{eg_c2},  the latter induction is not birational.
Hence, $\orb^{C_G(s_1)}_u$ is not birationally induced from $(M, \{1\})$ and $\orb^G_{s_1 u} \not\subset \overline{J(\tau)}^{bir}$.
This argument can be repeated, up to reordering the decomposition into simple groups, for $C_G(s_{-1}) \simeq {\rm Sp}_2(\C) \times {\rm Sp}_4(\C)$ and $v \in \ind_M^{C_G(s_{-1})} \{1\}$.
Therefore, the birational sheet $\overline{J(\tau)}^{bir} = J(\tau)$ does not contain any isolated class.
\end{example}

\begin{lemma} \label{lem_shbirsh}
Let $\tau = (M, Z(M)^\circ s, \orb^M) \in \mathscr{D}(G)$ with $\orb^M \in \mathcal{U}_M/M$ rigid. Then the birational sheet $\overline{J(\tau)}^{bir}$ is dense in the sheet $\overline{J(\tau)}^{reg}$ and $\overline{J(\tau)}^{reg}$ is the unique sheet of $G$ containing $\overline{J(\tau)}^{bir}$.

Moreover, $\overline{J(\tau)}^{bir}=\overline{J(\tau)}^{reg}$ if and only if all isolated classes $G \cdot (r\ind_{M}^{C_G(r)} \orb^M)$ with $r \in Z(M)^\circ s$ are birationally induced.
\end{lemma}

\begin{proof}
We have $\tau \in \mathscr{BB}(G)$, since $\orb^M$ is in particular birationally rigid and $\overline{J(\tau)}^{bir}$ is open and dense in the irreducible set $\overline{ J(\tau)}^{reg}$ by Proposition \ref{prop_birjord}. Suppose $S$ is a sheet of $G$ with $\overline{J(\tau)}^{bir} \subset S$. Then the closure of $\overline{J(\tau)}^{bir}$ equals $\overline{J(\tau)} \subset \overline{S}$, so $\overline{J(\tau)}^{reg} = S$.

The last assertion follows from Lemma \ref{lem_bir_strata} and Proposition \ref{prop_birzl_gp}.
\end{proof}

\begin{corollary}
Let $G$ be simple simply connected of type $\mathsf{A}$. Then all sheets are birational sheets. In particular, sheets of $G$ form a partition.
\end{corollary}

\begin{proof}
This follows from Example \ref{eg_type_a} and Lemma \ref{lem_shbirsh} and Theorem \ref{thm_partition}.
\end{proof}

\begin{remark} \label{rk_lusztig}
We claim that Lusztig's strata defined in \cite{lustrata} are disjoint unions of birational sheets.
This follows from \cite[Proof of Theorem 2.1]{CarnoBul}: it is proven therein that if $J \in \mathscr{J}(G)$ lies in a stratum, then $\overline{J}^{reg}$ lies in that stratum. Since $\overline{J}^{bir} \subset \overline{J}^{reg}$, we get that strata are unions of birational closures of Jordan classes. By taking maximal sets with respect to inclusion in this decomposition, we conclude our claim.
\end{remark}

\subsection{Weakly birational sheets} \label{s_wbs}
We would like to cast some light on the case in which $G$ is not simply connected.
Recall that in this case, birational induction and weakly birational induction are two distinct concepts so that $\bir ( Z(M)^\circ s , \orb^M )$ as in \eqref{eq_centre} can be a proper subset of $\wbir( Z(M)^\circ s , \orb^M )$ as in \eqref{eq_wcentre}.

In \S \ref{ss_bcjc} and \S \ref{ss_bs}, we treated Jordan classes as the ``building blocks'' of birational closures and birational sheets.
Recall Example \ref{eg_psl2} and retain notation therein:  for $\overline{G} = {\rm PSL}_2(\C)$, the Jordan class of regular semisimple elements $J(\tau) = \overline{G} \cdot (\overline{T} \setminus \{\bar{e}\})$ does not consist of all birationally induced conjugacy classes.
This implies it is not possible to extend directly constructions and proofs of results in  \S \ref{ss_bcjc} and \S \ref{ss_bs} by requiring that classes are birationally induced in the sense of Definition \ref{defn_birind} (a). 
Nonetheless, the results  in \S \ref{ss_noname} hold for $G$ not necessarily simply connected, therefore we give the following definition as a possible generalization.
\begin{definition} \label{defn_weak}
Let $su \in G$, the \emph{weakly birational closure} of $J(su)$ is:
$$\overline{J(su)}^{wbir} \coloneqq \bigcup_{z \in \wbir(Z(C_G(s)^\circ)^\circ s, \orb^{C_G(s)^\circ}_u)} G \cdot \left( z \ind_{C_G(s)^\circ}^{C_G(z)^\circ} \orb^{C_G(s)^\circ}_u \right).$$
\end{definition}

\begin{remark} \label{rk_wbc_ext}
All results concerning birational closures of Jordan classes for $G$ simply connected from  \S \ref{ss_bcjc} can be restated for weakly birational closures for any connected reductive group $G$. In particular, for $J \in \mathscr{J}(G)$, the weakly birational closure $\overline{J}^{wbir}$  is open  in $\overline{J}$.
\end{remark}

Similarly, for any connected reductive group $G$, it still makes sense to introduce $\mathscr{BB}(G) \subset \mathscr{D}(G)$ as in \eqref{eq_bbg}.
\begin{definition} \label{defn_weakbirsh}
For $\tau \in \mathscr{BB}(G)$, we call $ \overline{J(\tau)}^{wbir}$ the \emph{weakly birational sheet} associated to $\tau$.
\end{definition}

\begin{remark} \label{rk_wbs_ext}
All results concerning birational sheets proven in  \S \ref{ss_bs} for $G$ simply connected can be restated for weakly birational sheets in the case of any connected reductive group $G$; in particular, any connected reductive group is partitioned into its weakly birational sheets.\footnote{The proofs of the results in \S \ref{ss_bcjc} and \S \ref{ss_bs} easily adapt to this case: one needs to substitute all centralizers of semisimple elements with their identity component.}
\end{remark}

\section{Local geometry of birational closures} \label{s_locgeom}
We start this section with definitions and results from algebraic geometry which will be useful for our purposes.
Following terminology of \cite[§1.7]{Hess}, two pointed varieties $(X, x)$ and $(Y, y)$ are said to be \emph{smoothly
equivalent} if there exist a pointed variety $(Z, z)$ and two smooth maps $\phi \colon Z \to X
$ and $\psi \colon  Z \to Y$ such that $\phi(z) = x$ and $\psi(z) = y$.
In this case we write $(X,x) \sim_{se} (Y,y)$.
By \cite[Remark 2.1]{KP81}, if $\dim Y = \dim X + d$, then $(X, x) \sim_{se} (Y, y)$ if and only if
$(X \times \mathbb{A}^d , (x, 0))$ and $(Y, y)$ are locally analytically isomorphic.
Smooth equivalence is an equivalence relation on pointed varieties and it preserves
the properties of being unibranch, normal or smooth. 
For any algebraic variety $X$, denote by $X^{an}$ the associated analytic space.

\begin{lemma} \label{lem_normalization}
Let $X$ and $Y$ be complex algebraic varieties with $\dim X = \dim Y + d$. Let $X$ and $Y$ be unibranch at $x \in X$ and at $y \in Y$, respectively. Suppose $(X, x) \sim_{se} (Y, y)$. Let $\psi_X \colon \widetilde{X} \to X$ and $\psi_Y \colon \widetilde{Y} \to Y$ be the normalizations of $X$ and $Y$, respectively. Let $\tilde{x} \in \widetilde{X}$ and $\tilde{y} \in \widetilde{Y}$ with $\psi_X(\tilde{x}) = x$ and $\psi_Y(\tilde{y}) = y$, respectively. Then $(\widetilde{X}, \tilde{x}) \sim_{se} (\widetilde{Y}, \tilde{y})$.
\end{lemma}

\begin{proof}
By assumption, $(X^{an}, x)$ and $(Y^{an} \times \mathbb{A}^d, (y,0))$ are locally isomorphic as analytic pointed spaces.
Let $\widetilde{X^{an}}$ (resp. $\widetilde{Y^{an}}$) be the normalization of $X^{an}$ (resp. of $Y^{an}$).
 By \cite[§5, Satz 4]{Kuhlmann1961}, we have $\widetilde{X^{an}} = \widetilde{X}^{an}$ and $\widetilde{Y^{an}} = \widetilde{Y}^{an}$. Thus, $(\widetilde{X}^{an}, \tilde{x})$ is the analytic normalization of $(X,x)$ and $(\widetilde{Y}^{an} \times \mathbb{A}^d, (\tilde{y},0))$ is the analytic normalization of $(Y \times \mathbb{A}^d, (y,0))$. Hence, $(\widetilde{X}^{an}, \tilde{x})$   and $(\widetilde{Y}^{an} \times \mathbb{A}^d, (\tilde{y},0))$ are locally analytically isomorphic and this concludes the proof.
\end{proof}

Now we make use of the previously introduced instruments to describe the birational closure of a Jordan class around a unipotent element of $G$ connected and reductive.

\begin{lemma} \label{lem_unipot}
Let $L \leq G$ be a Levi subgroup and let $\tau = (L, Z(L)^\circ, \orb^L) \in \mathscr{D}(G)$ and set $J \coloneqq J(\tau)$. Let $\mathfrak{l}:=\Lie(L)$ and $\orb^{\mathfrak{l}} \in \mathcal{N}_\mathfrak{l}/L$ such that $\exp(\orb^{\mathfrak{l}}) = \orb^L$. Let $\mathfrak{J} \coloneqq \mathfrak{J}(\mathfrak{l}, \orb^\mathfrak{l}) \subset \mathfrak{g}$ be defined as in \eqref{eq_deco_class}.
Let $\nu \in \mathcal{N}$ and $v \coloneqq \exp (\nu) \in \mathcal{U}$. Then:
\begin{enumerate}[noitemsep, nolistsep]
\item[(i)]  $v \in \overline{J}^{wbir}$ if and only if $\nu \in \overline{\mathfrak{J}}^{bir}$;
\item[(ii)] if $v \in \overline{J}^{wbir}$, then $(\overline{J}^{wbir}, v) \sim_{se} (\overline{\mathfrak{J}}^{bir}, \nu)$.
\end{enumerate}
\end{lemma}

\begin{proof}
(i) We have $v \in \overline{J}^{wbir}$ if and only if $\orb_v^G = \ind_L^G \orb^L$ is birationally induced from $(L, \orb^L)$ (equiv. weakly birationally induced, see Remark \ref{rk_coincide}), if and only if $\orb^\mathfrak{g}_{\nu} = \ind_\mathfrak{l}^\mathfrak{g} \orb^\mathfrak{l}$ is birationally induced from $(\mathfrak{l}, \orb^\mathfrak{l})$ (see Remark \ref{rk_ind_isogeny} (iii)), i.e. $\nu \in \overline{\mathfrak{J}}^{bir}$.

(ii) For $v \in \overline{J}^{wbir}$, we have $(\overline{J}^{wbir}, v) \sim_{se} (\overline{J}, v)$, since $\overline{J}^{wbir}$ is open in $\overline{J}$ by Remark \ref{rk_wbc_ext}.
Moreover, $ (\overline{J}, v) \sim_{se}  (\overline{\mathfrak{J}}, \nu)$ thanks to \cite[Corollary 5.3]{ACE}.
Now, (i) implies $\nu \in \overline{\mathfrak{J}}^{bir}$, and $(\overline{\mathfrak{J}}, \nu) \sim_{se} (\overline{\mathfrak{J}}^{bir},\nu)$, as $\overline{\mathfrak{J}}^{bir}$ is open in $\overline{\mathfrak{J}}$.
\end{proof}

From this point up to the end of the paper, we will assume $G$ semisimple and simply connected.
We begin with a general result on the fact that the geometry of birational closures is constant along Jordan classes.
\begin{lemma} \label{lem_const_geom}
Let $J'$ and $J$ be Jordan classes in $G$ semisimple simply connected.
Let $J' \subset \overline{J}^{bir}$ and let $rv, r'v' \in J'$.
Then $(\overline{J}^{bir}, rv) \simeq_{se} (\overline{J}^{bir}, r'v')$.
\end{lemma}

\begin{proof}
This follows from Proposition \ref{prop_birjord} and \cite[Corollary 4.5]{ACE}, as  $\overline{J}^{bir}$ is open in $\overline{J}$.
\end{proof}

In the remainder, we will switch from Jordan classes in $G$ to Jordan classes in a reductive subgroup $M \leq G$ and back. We will use the blackboard bold typeface $\mathbb{J}$ to denote a Jordan class in $G$ and the usual typeface $J$ to denote Jordan classes in $M$. The (regular, resp. birational) closure of $\mathbb{J}$ will be implied to be performed in the group $G$ and the (regular, resp. weakly birational) closure of $J$ will be implied to be performed in the group $M$, therefore we will use previously introduced notations ($\overline{\mathbb{J}}, \overline{\mathbb{J}}^{reg}, \overline{\mathbb{J}}^{bir}, \overline{J}, \overline{J}^{wbir}$, \dots).

Let $su \in G$, with $s \in T$ and let $\mathbb{J} \coloneqq \mathbb{J}(su) \in \mathscr{J}(G)$.
We set the following notation: $C \coloneqq C_G(s) = C_G(s)^\circ$ and $Z \coloneqq Z(C)$.
Let $rv \in \overline{\mathbb{J}}$ with $r \in Z^\circ s \subset T$ and let
$M \coloneqq C_G(r) = C_G(r)^\circ$. Since $r \in Z^\circ s$, we have that $C$ is a Levi subgroup of $M$ by Lemma \ref{lem_itsalevi}.
Set $\mathcal{J}_{\mathbb{J},rv} \coloneqq \{J_\ell \in \mathscr{J}(M) \mid J_\ell \subset \mathbb{J} \mbox{ and } rv \in \overline{J_\ell} \}$.

The result in \cite[Proposition 4.3]{ACE} states that if $rv \in \overline{\mathbb{J}}$, then $(\overline{\mathbb{J}}, rv) \sim_{se} (\bigcup_{\mathcal{J}_{\mathbb{J},rv}} r^{-1} \overline{J}_\ell, v)$.
Moreover, $rv \in \overline{\mathbb{J}}^{reg}$ if and only if $rv \in \overline{J}_\ell^{reg}$ for all $J_\ell \in \mathcal{J}_{\mathbb{J},rv}$, and in this case $(\overline{\mathbb{J}}^{reg}, rv) \sim_{se} (\bigcup_{\mathcal{J}_{\mathbb{J},rv}} r^{-1}\overline{J}^{reg}_\ell, v)$.
Finally, if $\overline{\mathbb{J}}^{reg}$ is a sheet in $G$, then all $\overline{J}_\ell^{reg} \in \mathcal{J}_{\mathbb{J},rv}$ are sheets in $M$, see \cite[Remark 6.4 (2)]{ACE}. 

Suppose now that $rv \in \overline{\mathbb{J}}^{bir}$. We will prove that $(\overline{\mathbb{J}}^{bir}, rv)$ and $(\bigcup_{\mathcal{J}_{\mathbb{J},rv}} \overline{J}_\ell^{wbir}, rv)$ are smoothly equivalent.

\begin{lemma} \label{lem_bir_smeq}
Let $G$ be semisimple simply connected, let $su \in G$ with $s \in T$ and $\mathbb{J}(su) \eqqcolon \mathbb{J} \in \mathscr{J}(G)$ and let $rv \in \overline{\mathbb{J}}$.
Then $rv \in \overline{\mathbb{J}}^{bir}$ if and only if $rv \in \overline{J}_\ell^{wbir}$ for all $J_\ell \in \mathcal{J}_{\mathbb{J},rv}$.
In this case,
$$(\overline{\mathbb{J}}^{bir}, rv) \sim_{se} \left( \bigcup_{\mathcal{J}_{\mathbb{J},rv}} r^{-1} \overline{J}_\ell^{wbir}, v \right).$$
\end{lemma}

\begin{proof}
We can assume $r \in Z^\circ s$.
Denote with $$\mathcal{A}_{\mathbb{J},rv} \coloneqq \{(C_\ell, Z(C_\ell)^\circ z_\ell, \orb^{C_\ell}) \in \mathscr{D}(M) \mid r \in Z(C_\ell)^\circ z_\ell \mbox{ and } \orb^M_v \subset \overline{\ind_{C_\ell}^{M} \orb^{C_\ell}} \}.$$

If $J_\ell \in \mathcal{J}_{\mathbb{J},rv}$, there exists $\tau_\ell \in \mathcal{A}_{\mathbb{J},rv}$ 
such that $J(\tau_\ell) = J_\ell$, see \cite[Theorem 4.4]{ACE}.
Note that $\tau \coloneqq (C, Z^\circ s, \orb^C_u) \in \mathcal{A}_{\mathbb{J},rv}$.
By Lemma \ref{lem_itsalevi}, since $r \in Z(C_\ell)^\circ z_\ell$, we have that $C_\ell$ is a Levi subgroup of $M$.
Hence, by \cite[Lemma 3.10]{CE1} $C_\ell$ is a pseudo-Levi of $G$, so that $\tau_\ell \in \mathscr{D}(G)$.
By definition of $\mathcal{J}_{\mathbb{J},rv}$, we have that $\mathbb{J}(\tau_\ell) = \mathbb{J}$ for every $\tau_\ell \in \mathcal{A}_{\mathbb{J},rv}$.
Suppose that $rv \in \overline{\mathbb{J}}^{bir}$.
Then, $rv \in \overline{\mathbb{J}(\tau_\ell)}^{bir}$ for all $\tau_\ell \in \mathcal{A}_{\mathbb{J},rv}$, i.e.
$$G \cdot (r \orb^M_v) \subset \bigcup_{z \in \bir(Z(C_\ell)^\circ z_\ell, \orb^{C_\ell})} 
G \cdot (z \ind_{C_\ell}^{{C_G(z)}} \orb^{C_\ell}).$$
Since $G$ is simply connected and $r \in Z^\circ s$, this is equivalent to $r \orb^M_v = \orb^M_{rv}$ being birationally induced from $(C_\ell, r, \orb^{C_\ell}) \in \mathscr{B}(M)$ (equiv. weakly birationally induced, see Remark \ref{rk_coincide}) for all $\tau_\ell \in \mathcal{A}_{\mathbb{J},rv}$, i.e. $rv \in \overline{J}^{wbir}_\ell$, for all $\tau_\ell \in \mathcal{A}_{\mathbb{J},rv}$.
For the other implication, remark that $J(\tau) \in \mathcal{J}_{\mathbb{J},rv}$, and if $\orb^M_{rv} \subset \overline{J(\tau)}^{wbir}$, then $\orb^G_{rv} = G \cdot (\orb^M_{rv}) \subset \overline{\mathbb{J}}^{bir}$.

For the last part of the statement, let $rv \in \overline{\mathbb{J}}^{bir} \subset \overline{\mathbb{J}}^{reg}$.
Observe that $r \in Z(M)$ implies $r^{-1}\overline{J}_\ell^{wbir} = \overline{r^{-1}J}_\ell^{wbir}$.
By \cite[Proposition 4.3]{ACE}, $(\overline{\mathbb{J}}^{reg}, rv) \sim_{se} (\bigcup_{\mathcal{J}_{\mathbb{J},rv}} r^{-1}\overline{J}^{reg}_\ell, v)$.
Since $\overline{\mathbb{J}}^{bir}$ is open in $\overline{\mathbb{J}}^{reg}$, we have $(\overline{\mathbb{J}}^{reg}, rv) \sim_{se} (\overline{\mathbb{J}}^{bir}, rv)$;
similarly each $\overline{J}_\ell^{wbir}$ is open in $\overline{J}_\ell^{reg}$ (see Remark \ref{rk_wbc_ext}), hence $(\bigcup_{\mathcal{J}_{\mathbb{J},rv}} \overline{J}^{reg}_\ell,rv) \sim_{se} (\bigcup_{\mathcal{J}_{\mathbb{J},rv}} \overline{J}^{wbir}_\ell, rv)$.
Transitivity yields: $(\overline{\mathbb{J}}^{bir}, rv) \sim_{se}(\bigcup_{\mathcal{J}_{\mathbb{J},rv}} \overline{J}_\ell^{wbir}, rv) \sim_{se}(\bigcup_{\mathcal{J}_{\mathbb{J},rv}} r^{-1}\overline{J}_\ell^{wbir}, v)$. 
\end{proof}

We prove the main result of this section, which reduces the local study of a birational sheet in $G$ semisimple simply connected to the study of a neighbourhood of a unipotent element in some connected reductive subgroup of $G$.

\begin{theorem} \label{thm_only_one}
Let $G$ be semisimple simply connected and $\tau = (C, Z^\circ s, \orb^C) \in \mathscr{BB}(G)$. For $rv \in \overline{\mathbb{J(\tau)}}^{bir}$ with $r \in Z^\circ s$, we have
$$\left( \overline{\mathbb{J(\tau)}}^{bir}, rv \right) \sim_{se} \left( r^{-1}\overline{J(\tau)}^{wbir}, v \right).$$
\end{theorem}

\begin{proof}
The statement follows from Lemma \ref{lem_bir_smeq}, provided that $\lvert \mathcal{J}_{\mathbb{J},rv} \rvert = 1$ if $\overline{\mathbb{J}}^{bir}$ is a birational sheet.
Since $\orb^C$ is birationally rigid in $C$, also $ \orb^{C_\ell} $ is birationally rigid in $C_\ell$ for all $\ell$ such that $J_\ell \in \mathcal{J}_{\mathbb{J},rv}$, as $(C, \orb^C)$ and $(C_\ell, \orb^{C_\ell})$ are conjugate in $G$ (see Remark \ref{rk_bir_rig}).
Thus each $\overline{J_\ell}^{wbir}$ is a  weakly birational sheet of $M$ containing $rv$.
By Remark \ref{rk_wbs_ext}, $M$ is partitioned into its weakly birational sheets, hence  $\mathcal{J}_{\mathbb{J},rv} = \{J(\tau)\}$.
\end{proof}

We conclude with three results: the first one compares the local structures of a birational sheet of $G$ semisimple simply connected and of a birational sheet in a reductive Lie subalgebra of $\mathfrak{g}$, the last two are applications.

\begin{theorem}\label{thm_locbirsheet}
Let $G$ be semisimple simply connected, let $\tau = (C, Z^\circ s, \orb^C) \in \mathscr{BB}(G)$ and let $rv \in \overline{\mathbb{J(\tau)}}^{bir}$ with $r \in Z^\circ s$.
Then $$(\overline{\mathbb{J}(\tau)}^{bir}, rv) \sim_{se} (\overline{\mathfrak{J}_{\mathfrak{c_g}(r)}(\mathfrak{c}, \orb^{\mathfrak{c}})}^{bir}, \nu)$$
where $\mathfrak{c} \coloneqq \Lie(C)$, the orbit $\orb^\mathfrak{c} \in \mathcal{N}_\mathfrak{c}/C$ satisfies $\exp(\orb^{\mathfrak{c}}) = \orb^C$, the set $\overline{\mathfrak{J}_{\mathfrak{c_g}(r)}(\mathfrak{c}, \orb^{\mathfrak{c}})}^{bir}$ is a birational sheet in $\mathfrak{c}_\mathfrak{g}(r)$ and $\nu \in \mathcal{N}_{\mathfrak{c_g}(r)}$ is such that $\exp(\nu) = v$. 
\end{theorem}

\begin{proof}
This follows from Theorem \ref{thm_only_one}, Lemma \ref{lem_unipot} and the fact that $\orb^C \in \mathcal{U}_C/C$ is birationally rigid if and only if $\orb^\mathfrak{c} \in \mathcal{N}_\mathfrak{c}/C$ is so by Remark \ref{rk_bir_rig}.
\end{proof}

\begin{theorem} \label{thm_unibranch}
Let $G$ be semisimple simply connected. Let $\tau = (C, Z^\circ s, \orb^C) \in \mathscr{BB}(G)$. Then:
\begin{enumerate}[noitemsep, nolistsep]
\item[(i)] $\overline{\mathbb{J}(\tau)}^{bir}$ is unibranch;
\item[(ii)] the normalization of $\overline{\mathbb{J}(\tau)}^{bir}$ is smooth.
\end{enumerate}
\end{theorem}

\begin{proof}
(i) Let $rv \in \overline{\mathbb{J}(\tau)}^{bir}$. Theorem \ref{thm_locbirsheet}, from which we retain notation, implies $(\overline{\mathbb{J}(\tau)}^{bir}, rv) \sim_{se} (\overline{\mathfrak{J}_{\mathfrak{c_g}(r)} (\mathfrak{c}, \orb^{\mathfrak{c}})}^{bir}, \nu)$.
By Remark \ref{rk_bir_rig}, $\orb^\mathfrak{c}$ is birationally rigid and the statement follows from \cite[Theorem 4.4 (2)]{Lo2016} applied to the birational sheet $\overline{\mathfrak{J}_{\mathfrak{c_g}(r)} (\mathfrak{c}, \orb^{\mathfrak{c}})}^{bir}$. We remark that \cite[Theorem 4.4 (2)]{Lo2016} extends to the case of a reductive Lie algebra.

(ii) By (i), we have that $(\overline{\mathbb{J}(\tau)}^{bir}, rv) \sim_{se} (\overline{\mathfrak{J}_{\mathfrak{c_g}(r)} (\mathfrak{c}, \orb^{\mathfrak{c}})}^{bir}, \nu)$ are both unibranch.
Let $\psi_{\mathbb{J}} \colon \widetilde{\mathbb{J}} \to \overline{\mathbb{J}}^{bir}$ and $\psi_{\mathfrak{J}} \colon \widetilde{\mathfrak{J}} \to \overline{\mathfrak{J}_{\mathfrak{c_g}(r)} (\mathfrak{c}, \orb^{\mathfrak{c}})}^{bir}$ denote the normalization maps of $ \overline{\mathbb{J}}^{bir}$ and  $\overline{\mathfrak{J}_{\mathfrak{c_g}(r)} (\mathfrak{c}, \orb^{\mathfrak{c}})}^{bir}$, respectively.
Then $(\widetilde{\mathbb{J}}, \psi_\mathbb{J}^{-1}(rv)) \sim_{se} (\widetilde{\mathfrak{J}}, \psi_\mathfrak{J}^{-1}(\nu))$ by Lemma \ref{lem_normalization} and we conclude by \cite[Theorem 4.4 (2)]{Lo2016}.
\end{proof}

\begin{theorem} \label{thm_smooth}
Let $G$ be semisimple simply connected and let its simple factors have classical root system. Then the birational sheets of $G$ are smooth.
\end{theorem}

\begin{proof}
We prove that for $\tau = (C, Z^\circ s, \orb^C) \in \mathscr{BB}(G)$ and $rv \in \overline{\mathbb{J(\tau)}}^{bir}$ with $r \in Z^\circ s$, the birational sheet $\overline{\mathbb{J(\tau)}}^{bir}$ is smooth at $rv$.
By Theorem \ref{thm_locbirsheet}, from which we retain notation, we have  $(\overline{\mathbb{J}(\tau)}^{bir}, rv) \sim_{se} (\overline{\mathfrak{J}_{\mathfrak{c_g}(r)}(\mathfrak{c}, \orb^{\mathfrak{c}})}^{bir}, \nu)$.
The  algebra $\mathfrak{c_g}(r)$ is reductive, hence  $\overline{\mathfrak{J}_{\mathfrak{c_g}(r)}(\mathfrak{c}, \orb^{\mathfrak{c}})}^{bir}$ decomposes as a product of a vector space and some birational sheets in simple classical Lie subalgebras of $[\mathfrak{c_g}(r),\mathfrak{c_g}(r)]$, and all such objects are smooth (see \cite[Remark 4.10]{Lo2016}).
\end{proof}


\end{document}